\documentclass[12pt]{amsart}
\usepackage{latexsym}
\usepackage{amsmath,amsthm,lipsum}
\usepackage{mathtools}
\usepackage{amssymb,amsfonts,amscd}
\usepackage{mathrsfs}
\usepackage{graphicx}
\usepackage{multirow}
\usepackage{mathabx}
\usepackage{hyperref}
\usepackage{color}
\usepackage{cite}
\usepackage{algorithm}
\usepackage{todonotes}
\usepackage[dvipsnames]{xcolor}
\usepackage{changes}
\usepackage{hhline}
\usepackage{tabularx}
\usepackage{array}
\usepackage{siunitx}
\usepackage{algpseudocodex}
\usepackage{thmtools}
\usepackage{caption}
\usepackage{subcaption}
\usepackage{booktabs}
\usepackage{longtable}
\usepackage{placeins}
\usepackage{siunitx}
\usepackage{longtable}
\usepackage{pgfplots}
\pgfplotsset{compat=1.18}
\usepackage{tikz}
\usepackage{capt-of}
\usepackage{framed}

\usepackage[margin=1in]{geometry}
\usepackage{diagbox}
\usepackage{tikz}
\usepackage{caption}

\theoremstyle{plain} 
    \newtheorem{theorem}{Theorem}[section]
    \newtheorem{proposition}[theorem]{Proposition}

\theoremstyle{definition} 
    
    \newtheorem{example}[theorem]{Example}
    \newtheorem{remark}[theorem]{Remark}

    \newtheorem{problem}[theorem]{Problem}

\renewcommand\qed{{\hspace*{\fill}$\Box$\vskip12pt plus 1pt}}
\renewenvironment{proof}{{\noindent\bf Proof.\ }}{\qed}

\newcommand\suchthat{~|~}

\DeclareMathOperator*{\conj}{\rm conj}

\AtBeginEnvironment{example}{%
  \pushQED{\qed}%
}
\AtEndEnvironment{example}{\hfill $\diamond$\endexample}

\newcommand{\CC}{{\mathbb C}}

\newcommand{\NN}{{\mathbb N}}
\newcommand{\QQ}{{\mathbb Q}}
\newcommand{\RR}{{\mathbb R}}

\newcommand{\PP}{{\mathbb P}}

\usetikzlibrary{decorations.markings,arrows.meta}

\newcommand{\defn}[1]{{\color{blue}{#1}}}

\title[Common Real Secants
to Pairs of Real Twisted Cubic Curves]{Common Real Secants \\ 
to Pairs of Real Twisted Cubic Curves}

\begin{document}

\author[Aslam]{Saima Aslam} 
\address{S.~Aslam\\ 
         Department of Mathematics and Statistics\\ 
         University of Agriculture, Faisalabad\\ 
         Faisalabad\\  
         Pakistan} 
\email{saima.aslam@uaf.edu.pk}
\author[Faust]{Matthew Faust} 
\address{M.~Faust\\ 
         Department of Mathematics\\ 
         Michigan State University\\ 
         East Lansing\\ 
         Michigan \ 48824\\ 
         USA} 
\email{mfaust@msu.edu}
\urladdr{https://mattfaust.github.io/}
\author[Hauenstein]{Jonathan D. Hauenstein} 
\address{J. D.~Hauenstein\\ 
         Department of Applied and Computational Mathematics and Statistics\\ 
         University of Notre Dame\\ 
         Notre Dame\\ 
         Indiana \ 46556\\ 
         USA} 
\email{hauenstein@nd.edu}
\urladdr{https://hauenstein.phd}
\author[Lopez Garcia]{Jordy Lopez Garcia} 
\address{J.~Lopez Garcia\\ 
         Department of Applied and Computational Mathematics and Statistics\\ 
         University of Notre Dame\\ 
         Notre Dame\\ 
         Indiana \ 46556\\ 
         USA} 
\email{jlopezga@nd.edu}
\urladdr{https://jordylopez27.github.io/}
\author[Kagy]{Bryson Kagy} 
\address{B.~Kagy\\ 
         Department of Mathematics\\ 
         Texas State University\\ 
         San Marcos\\ 
         Texas \ 78666\\ 
         USA} 
\email{brysonkagy@gmail.com}
\urladdr{https://brysonkagy.github.io}
\author[Regan]{Margaret H.~Regan} 
\address{M. H.~Regan\\ 
         Department of Mathematics\\ 
         Grand Valley State University\\ 
         Allendale\\ 
         Michigan \ 49401\\ 
         USA} 
\email{reganmar@gvsu.edu}
\urladdr{https://www.margarethregan.com}
\author[Wampler]{Charles W. Wampler} 
\address{C. W.~Wampler\\ 
         Department of Applied and Computational Mathematics and Statistics\\ 
         University of Notre Dame\\ 
         Notre Dame\\ 
         Indiana \ 46556\\ 
         USA} 
\email{cwample1@nd.edu}
\urladdr{https://cwampler.com}
\author[Zhang]{Albert Zhang} 
\address{A.~Zhang\\ 
         Department of Mathematics\\ 
         University of California, Santa Cruz\\ 
         Santa Cruz\\ 
         California \ 95064\\ 
         USA} 
\email{azhang87@ucsc.edu}
\urladdr{https://azhang87.github.io}

\thanks{\texttt{Common Real Secants to Pairs of Real Twisted Cubic Curves} version \texttt{2.0}}

\begin{abstract}
\noindent
It is well established that a general pair of twisted cubic curves in complex 
projective space has ten common secant lines. 
First,
we show that the monodromy group of 
the ten common secant lines 
over the complex numbers
is the full symmetric group, demonstrating
that the common secant lines have
no special structure over the complex numbers.
We then investigate a novel question 
in real algebraic geometry: describe the possible collections of ten common real secant lines to a pair of real projective twisted cubic curves.
In addition to distinguishing between real and nonreal secant lines, we introduce a refinement of this classification which takes intersection points into account,
yielding totally real, partially real, and minimally real secant lines.  Using computational algebraic geometry as well as combinatorics, we show that for each~$k$ between~0 and 10, there exist pairs of real twisted cubic curves with exactly $k$ common totally real secant lines. 
We also obtain examples of 
real twisted cubics whose sets of common real secants cover a wide range of possibilities within our admissible classification
of common real secant lines.
\end{abstract}



\maketitle

\section{Introduction}

\defn{Twisted cubic curves} are smooth rational curves of degree three in $\PP^{3}(\CC)$. This class of curves have been studied for nearly two centuries (see~\cite{Little} by Little), and they are of particular interest to algebraic geometers, as they admit simple parameterizations while exhibiting subtle geometric behavior. 
As stated in Harris' book \textit{Algebraic Geometry}~\cite[Ex.~1.10]{HarrisBook},
``This is everybody's first example of a concrete variety that is not a hypersurface, linear space, or finite set of points.''

As a curve in $\PP^{3}(\CC)$, a twisted cubic curve has codimension two but is not a complete intersection since it cannot be defined by the vanishing of two polynomial equations;
three are required. 
For example, the \defn{standard twisted cubic} in $\PP^{3}(\CC)$ is the curve
\begin{equation}\label{eq:StandardTwisted}
C_0 = \bigl\{[s^3,s^2t,st^2,t^3]\mid[s,t]\in\PP^1(\CC)\bigr\}\subset\PP^3(\CC),
\end{equation}
and the ideal of polynomials vanishing on $C_0$ is
\begin{equation}\label{eq:IdealStandardTwisted}
I(C_0) = \langle x_1^2 - x_0 x_2,\,\,\, x_1 x_2 - x_0 x_3,\,\,\, x_1x_3 - x_2^2 \rangle\subset\CC[x_0,x_1,x_2,x_3].
\end{equation}
The dimension of the 
complex parameter space \defn{$\mathscr{H}$} of twisted cubic curves in~$\PP^3(\CC)$ is $12$~\cite[Lecture~12]{HarrisBook}.

\medskip
Given a twisted cubic curve $C\subset\PP^3(\CC)$, a line $\ell\subset\PP^3(\CC)$ is a \defn{secant line} of $C$ if $C$ and $\ell$ intersect in at least two distinct points, 
i.e., $\#(C\cap \ell)\geq 2$. In fact, this must be an equality
since a twisted cubic curve can intersect a line in at most two distinct points~\cite[\S3.1]{Arbarello}. Figure~\ref{fig:TwistedSecant} illustrates the standard twisted cubic 
$C_0$ together with a secant line $\ell$.
\begin{figure}[htbp]
\centering
\begin{picture}(150,120)
\put(30,0){\includegraphics[scale=0.6]{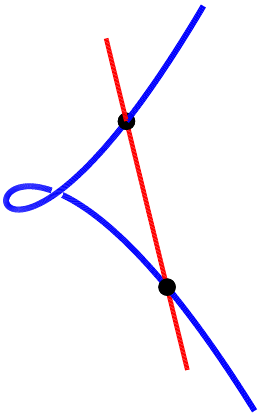}}
\put(80,90){$C_0$}
\put(50,95){$\ell$}
\end{picture}    
    \caption{An illustration 
    of the standard twisted cubic curve $C_0$ 
    together with a secant line $\ell$.}
    \label{fig:TwistedSecant}
\end{figure}

Since the closure of the set of secant lines to
a twisted cubic curve $C\subset\PP^3(\CC)$ is a surface in the four-dimensional Grassmannian of lines in $\PP^3(\CC)$, one expects that, given two general twisted cubic curves, the set of lines that are secants to both, i.e., \defn{common secant lines}, is finite. Indeed, it is a classical enumerative problem to count the number of such common secant lines to two general twisted cubic curves; see, e.g., \textit{3264 and All That} by Eisenbud and Harris~\cite[Ch.~3, Keynote~Question~(c)]{3264}, and \textit{Intersection Theory} by Fulton~\cite[Ex.~14.7.7(c)]{fulton-intersection-theory}.  This problem may be solved using techniques from Schubert calculus~\cite[Prop.~3.14]{3264} and dynamic projections~\cite[\S 3.5.2]{3264}, resulting in the following theorem.

\begin{theorem}\label{thm:10-complex-secant-lines}
Two general twisted cubic curves in $\PP^3(\CC)$ have ten common secant lines.
\end{theorem}

Two natural follow-up questions arise from this theorem and, to the best of our knowledge,
neither has previously been explored: determining the monodromy group of this enumerative problem and investigating its real analogue. These questions motivate our work.

As a pair of twisted cubic curves varies along a closed loop in the parameter space of general pairs, their common secant lines may be permuted. The permutations obtained in this way form the \defn{monodromy group} of this enumerative problem. In~\cite{Harris1}, Harris studied monodromy groups of enumerative problems, including the 27 lines on a smooth cubic surface, illustrating how these groups reveal geometric structure beyond an enumerative count; see also the survey by Sottile and Yahl in~\cite{SottileYahl}. 

\begin{problem}\label{problem:Monodromy}
What is the monodromy group of the ten common secant lines 
to two twisted cubic curves in $\PP^3(\CC)$?  
\end{problem}

By Theorem~\ref{thm:10-complex-secant-lines}, the monodromy group is a subgroup of $S_{10}$, the symmetric group on ten elements. Our first result is 
that it is indeed 
the full symmetric group.

\begin{theorem}\label{thm:Monodromy}
The monodromy group of the ten common secant lines to two twisted cubic curves in $\PP^3(\CC)$ is $S_{10}$.    
\end{theorem}

We prove Theorem~\ref{thm:Monodromy} via computational algebraic geometry. First, we utilize the software systems 
\texttt{Bertini}~\cite{Bertini}, developed by Bates, Hauenstein, Sommese and Wampler, and \texttt{Mathematica}~\cite{Mathematica} to heuristically 
determine two monodromy loops whose corresponding monodromy actions appear to generate $S_{10}$.  
Once identified, we use the certified path-tracking software \texttt{CertifiedHomotopyTracking.jl}~\cite{CertifiedHomotopyTracking} by Lee to prove this is indeed the~case.

\bigskip
We now consider formulating the real analogue of Theorem~\ref{thm:10-complex-secant-lines}. A fundamental question in real enumerative geometry is to determine the possible numbers of real solutions to an enumerative problem (see, e.g.,~\cite{SottileRealEnumerative} by Sottile). A curve in $\PP^3(\CC)$ is said to be \defn{real} if it has a set of defining equations with real coefficients. This defines both a real twisted cubic curve and a real line in $\PP^3(\CC)$. For example, the standard twisted cubic curve~$C_0$ in~\eqref{eq:StandardTwisted} is real as observed from the defining equations in~\eqref{eq:IdealStandardTwisted}.

Suppose that $C_1$ and $C_2$ are general twisted cubic curves that are real and $\ell$ is a real line that is a secant line to both curves, i.e., a \defn{common real secant line}.
Considering the set of intersections $C_1\cap \ell$ and $C_2\cap \ell$, we note that there are only three possibilities:
\begin{itemize}
    \item $\ell$ is \defn{totally real} if both $C_1\cap \ell$ and $C_2\cap \ell$ consist of two real points;
    \item $\ell$ is \defn{partially real} if one of $C_1\cap \ell$ or $C_2\cap \ell$ consists of two real points and the other consists of a pair of nonreal complex conjugate points; and
    \item $\ell$ is \defn{minimally real} if both $C_1\cap \ell$ and $C_2\cap \ell$ consist of a pair of nonreal complex conjugate points.
\end{itemize}
Thus, a common real secant line must be either totally real, partially real, or minimally real. A common secant line that is not real is called a \defn{common nonreal secant line}. 
This classification yields
the following real analogue to~\cite[Ch.~3, Keynote~Question~(c)]{3264} and~\cite[Ex.~14.7.7(c)]{fulton-intersection-theory}. 

\begin{problem}\label{problem:real}
Given two general twisted cubic curves in $\PP^3(\CC)$ that are real, how many totally real, partially real,
minimally real, and nonreal common secant lines are possible?
\end{problem}

\noindent As an initial answer to Problem~\ref{problem:real}, we show that there is no constraint on the number of totally real lines.
\begin{theorem}\label{thm:totally-real}
    For each $k\in \{0,1,\dots,10\}$, there exists real twisted cubic curves $C_{1}^{(k)},C_{2}^{(k)}$ 
    having $10$ common secant lines
    with exactly $k$ common totally real secant lines.
\end{theorem}
\noindent We prove Theorem~\ref{thm:totally-real} by exhibiting specific curves $C_{1}^{(k)},C_{2}^{(k)}$ for each $k$ (see Section~\ref{subsec:totally-real}). 

By Theorem~\ref{thm:10-complex-secant-lines}, for a generic pair of twisted cubics that are real, the sum of the numbers of totally real, partially real, minimally real, and nonreal common secant lines must be~$10$. Thus, we can represent these counts using the $3$-tuple \defn{$(n_t,n_p,n_m)$} where \defn{$n_t$}, \defn{$n_p$}, and \defn{$n_m$} are the number of totally real, partially real, and minimally real common secant lines, respectively. The number of common real secant lines will be denoted by \defn{$n_{\RR}$} $\vcentcolon = n_t{+}n_p{+}n_m$
with the number of common nonreal secant lines
being $10-n_{\RR}$.
In particular, $n_t,n_p,n_m\in\{0,1,\dots,10\}$
such that $n_{\RR}=n_t+n_p+n_m$
is an even number between $0$ and $10$,
as common nonreal secant lines must arise in complex conjugate pairs. By a combinatorial argument (Proposition~\ref{prop:admissible-s}), there are a total of $161$ distinct $3$-tuples $(n_t,n_p,n_m)$ satisfying these conditions, which we call \defn{admissible $3$-tuples}. An admissible $3$-tuple $(n_{t},n_{p},n_{m})$ is \defn{realizable} if there exist real twisted cubic curves $C_{1},C_{2}$ 
having $10$ common secant lines 
with exactly $n_{t},n_{p}$, and $n_{m}$ totally real, partially real, and minimally real common secant lines, respectively. We say $C_{1},C_{2}$ \defn{realize}~$(n_{t},n_{p},n_{m})$, and $(n_{t}, n_{p}, n_{m})$ is \defn{realized} by $C_{1}$ and $C_{2}$. 

Our main result is the following theorem. 

\begin{theorem}\label{thm:main}
Given an admissible 3-tuple $(n_{t}, n_{p}, n_{m})$ with $n_{\RR} = n_t{+}n_p{+}n_m$ such that
\begin{equation*}
\begin{split}
n_{\RR}    & \in \{0,2,4,6\},\\
\text{or} \quad n_{\RR} & = 8 \quad \text{and} \quad (n_{t},n_{p}) \notin \{(1,0),(1,1)\},\\
\text{or} \quad n_{\RR} & = 10 \quad \text{and} \quad n_{t}\geq 5,\\
\text{or} \quad n_{\RR} & = 10 \quad \text{and} \quad n_{t} \geq 2, n_{m} \geq 3, (n_{t},n_{m}) \neq (4,6),
\end{split}
\end{equation*}
there exists real twisted cubic curves $C_{1},C_{2}$ that realize $(n_{t}, n_{p}, n_{m})$. 
\end{theorem}

\begin{figure}[htbp]
    \begin{picture}(100,180)
    \put(-120,20){\includegraphics[scale=0.35]{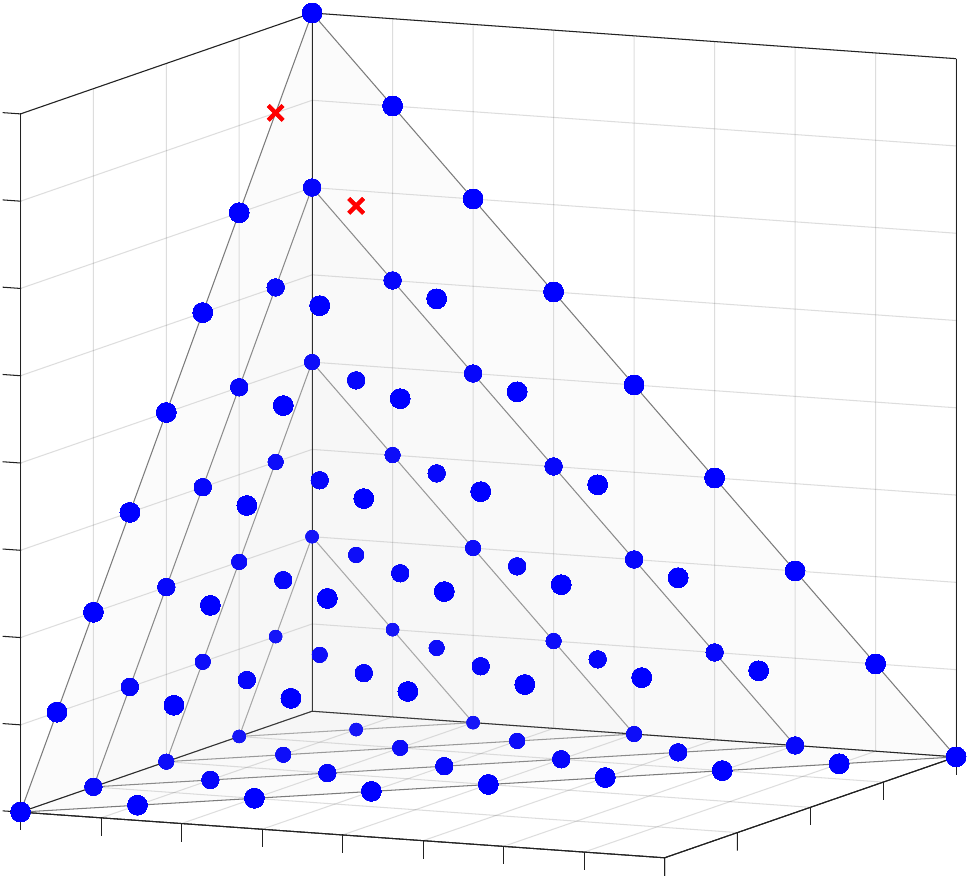}}
    \put(-130,15){8}
    \put(-135,140){8}
    \put(-10,5){8}
    \put(-140,80){$n_{m}$}
    \put(-70,10){$n_{p}$}
    \put(15,15){$n_{t}$}
    \put(80,20){\includegraphics[scale=0.35]{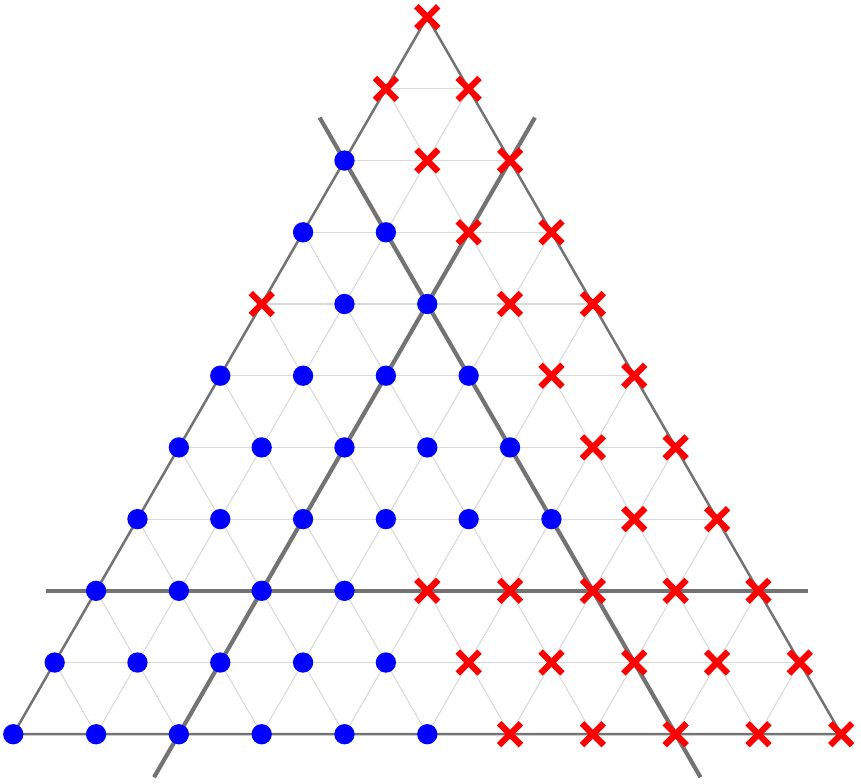}}
    \put(133,160){$(0,0,10)$}
    \put(215,10){$(0,10,0)$}
    \put(50,10){$(10,0,0)$}
    \put(170,138){$n_{p} = 2$}
    \put(100,138){$n_{t} = 2$}
    \put(220,50){$n_{m} = 2$}
    \end{picture}
    \caption{Realized 3-tuples are plotted with \textcolor{blue}{$\bullet$} and admissible but not yet realized 3-tuples are plotted with \textbf{\textcolor{red}{$\times$}}. The figure on the left shows the 3-tuples with $n_{\RR}\in \{0, 2,4,6,8\}$. The ternary plot on the right shows the 3-tuples with $n_{\RR} = 10$, with darker lines indicating $n_{t} = 2$, $n_{p} = 2$, and $n_{m} = 2$.}
    \label{fig:3-tuples-plot}
\end{figure}
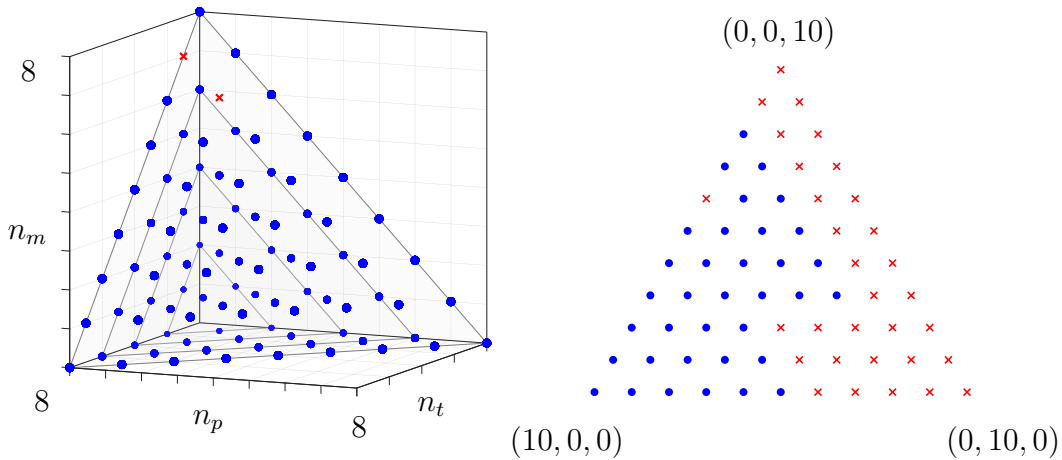

\begin{remark}
    The 128 admissible 3-tuples in Theorem~\ref{thm:main} are listed in Table~\ref{tab:3-tuples} and illustrated in Figure~\ref{fig:3-tuples-plot}. It remains an open problem to determine whether admissible 3-tuples
    of the form
    \mbox{$(0,n_{p},10-n_{p})$}
    for $0\leq n_{p}\leq 10$
    and \mbox{$(1,n_{p},9-n_{p})$}
    for $0\leq n_{p}\leq 9$
    are realizable.
\end{remark}

We prove Theorem~\ref{thm:main} by utilizing
computational algebraic geometry to exhibit examples of real twisted cubic curves realizing such admissible 3-tuples. We explore \defn{$\mathscr{H}(\RR)$}, the parameter space of real twisted cubic curves, sampling specific parameter points and using \texttt{Bertini} to heuristically determine the number of totally real, partially real, and minimally real common secant lines. Then, we utilize the software system \texttt{alphaCertified}~\cite{alphaCertified}, developed by Hauenstein and Sottile, to prove that those counts are indeed correct. Our code and sample parameter points are contained in a {\tt GitHub} repository~\cite{github-repo-AMS-MRC-2025}. Figure~\ref{fig:3-tuples-plot}
shows the admissible and realized $3$-tuples 
with Table~\ref{tab:3-tuples} in 
Appendix~\ref{sec:Table} listing 
the 3-tuples explicitly.

\medskip
The rest of the paper is organized as follows. Section~\ref{sec:parameterized-polynomial} considers the formulation of a 
common secant to two twisted cubic curves as a parameterized polynomial system, which will be used to solve Problem~\ref{problem:Monodromy} and to investigate Problem~\ref{problem:real}. Section~\ref{sec:MonodromyGroup} proves Theorem~\ref{thm:Monodromy} by computing two elements contained in 
the monodromy group that generate $S_{10}$. Section~\ref{sec:real-secants} proves Theorems~\ref{thm:totally-real} and~\ref{thm:main}, building on necessary conditions from Section~\ref{sec:parameterized-polynomial} (Proposition~\ref{prop:admissible-s}) with computational algebraic geometry exhibiting an example for each admissible $3$-tuple realized. Concluding remarks and open problems are provided in Section~\ref{sec:discussion-and-open-problems}.

\section{
Secants of twisted cubics
}\label{sec:parameterized-polynomial}

Our proofs of Theorems~\ref{thm:Monodromy},~\ref{thm:totally-real}, and~\ref{thm:main} will use a formulation of the problem which is intrinsic to the geometry of the parameter space $\mathscr{H}$ of twisted cubic curves~\cite[Example~12.9]{HarrisBook}. Using a parametrization of the standard twisted cubic $C_{0}$, we describe common secant lines via rank conditions.
To that end, we regard vectors in $\CC^{4}$ (and, by extension, points in~$\PP^{3}(\CC)$) as column vectors. Let $C_{0}$ denote the standard twisted cubic as in~\eqref{eq:StandardTwisted}
and $v(a) \vcentcolon = (1,a,a^{2},a^{3})^{T} \in \CC^{4}$. 
When $s\neq 0$, 
we have $[s,t] = [1,a]\in\PP^1(\CC)$ 
where $a = t/s$.~Hence,
\begin{equation*}
    C_0 = \overline{\{[v(a)] \mid a\in \CC\}}
    = \{[v(a)] \mid a\in \CC\} \cup \{[0,0,0,1]\}.
\end{equation*}

It is known that every twisted cubic in $\PP^{3}(\CC)$ is the image of $C_{0}$ under a projective linear transformation~\cite[Example 1.14]{HarrisBook}. Let \defn{$\mathrm{GL}_{4}(\CC)$} be the group of invertible $4\times 4$ matrices with entries in $\CC$. Given a matrix $M \in \mathrm{GL}_{4}(\CC)$, the action of $M$ on $\PP^{3}(\CC)$ given by
\begin{equation*}
    M\cdot [x] \vcentcolon  = [Mx], \quad 0\neq x \in \CC^{4}
\end{equation*}
is well-defined and scaling $M$ does not change the action. Now, let $C_{1}$ be the image of $C_{0}$ under this projective transformation. Thus, 
\begin{equation*}
C_{1} \vcentcolon = M(C_{0}) = \overline{\bigl\{[Mv(s)] \in \PP^{3}(\CC) \mid s\in \CC\bigr\}}.
\end{equation*}
In particular, $C_{1}$ only depends on the equivalence class $[M]\in$ \defn{$\mathrm{PGL}_{4}(\CC)$} $\vcentcolon = \mathrm{GL}_{4}(\CC) / \CC^{\times}$
where $\CC^\times = \CC\setminus\{0\}$.
The group $\mathrm{PGL}_{4}(\CC)$, which describes all projective linear transformations of~$\PP^{3}(\CC)$, is also identified with the Zariski open subset of $\PP(\CC^{4\times 4})\cong\PP^{15}(\CC)$ consisting of points $[M]$ with $\det(M) \neq 0$. 
Hence, we obtain a surjective map $\mathrm{PGL}_{4}(\CC) \to \mathscr{H}$ sending $[M] \in \mathrm{PGL}_{4}(\CC)$ to $M(C_{0})$. The stabilizer of $C_{0}$ under this action is isomorphic to $\mathrm{PGL}_{2}(\CC)$, e.g., see~\cite[Example 10.9]{HarrisBook}. For any $C\in \mathscr{H}$, the fiber over $C$ is a coset of this stabilizer. Thus $\mathscr{H} \cong \mathrm{PGL}_{4}(\CC) / \mathrm{PGL}_{2}(\CC)$, which has dimension $(4^2-1) - (2^2-1) = 12$.

\begin{remark}
When working with $[M] \in \mathrm{PGL}_{4}(\CC)$ we choose representatives $M \in \mathrm{GL}_{4}(\CC)$.
\end{remark}

Incidence conditions between $C_{0}$ and $C_{1}$ may be expressed in terms of the linear dependence of vectors in $\CC^{4}$. 
In fact, a line meeting 
$\{[v(a)] \mid a\in \CC\}$
and $\bigl\{[Mv(s)] \in \PP^{3}(\CC) \mid s\in \CC\bigr\}$
in two points each corresponds to a rank condition on the $4\times 4$ matrix whose columns are
\begin{equation*}
v(t_{1}), \quad v(t_{2}), \quad Mv(s_{1}), \quad Mv(s_{2})
\end{equation*}
where $t_1\neq t_2$ and $s_1\neq s_2$.
In particular, when $t_{1}\neq t_{2}$, the vectors 
$v(t_1)$, $v(t_2)$ are linearly independent and
\begin{equation}\label{eq:projective-secant-C0}
    \ell(t_{1},t_{2}) = \mathrm{Span}(v(t_{1}),v(t_{2})) \subset \PP^{3}(\CC)
\end{equation}
is a secant line to $C_0$.
The line $\ell(t_{1},t_{2})$ intersects
$\bigl\{[Mv(s)] \in \PP^{3}(\CC) \mid s\in \CC\bigr\}$ in two distinct points
if and only if there exist $s_{1}\neq s_{2}$ such that
\begin{equation*}
    Mv(s_{1}), Mv(s_{2}) \in \ell(t_{1},t_{2}).
\end{equation*}
Equivalently, for each $i=1,2$, the vectors $v(t_{1}),v(t_{2}),Mv(s_{i})^{T} \in \CC^{4}$ are linearly dependent:
\begin{equation}\label{eq:rank}
    \mathrm{rank}
    \begin{pmatrix}
    v(t_{1}) & v(t_{2}) & Mv(s_{i})  
    \end{pmatrix}\leq 2, \quad i=1,2.
\end{equation}
Since $t_1\neq t_2$ implies 
that $\mathrm{rank}
    \begin{pmatrix}
    v(t_{1}) & v(t_{2})
    \end{pmatrix} = 2$
with the leading $2\times 2$ block
being invertible, Condition~\eqref{eq:rank}
with $t_1\neq t_2$ is equivalent
to the vanishing of
the $3\times 3$ minors 
of $\begin{pmatrix}
    v(t_{1}) & v(t_{2}) & Mv(s_i) 
    \end{pmatrix}$
    using rows $1,2,3$ and $1,2,4$.
For convenience, write
 $M = (m_{ij})_{1\leq i,j\leq 4}$ and define 
\begin{equation}\label{eq:cubic-poly}
    p_{i}(s;M)\vcentcolon = m_{i1} + m_{i2}s + m_{i3}s^{2} + m_{i4}s^{3}, \quad i = 1,2,3,4.
\end{equation}
Then, computing these two minors
and factoring out the common factor 
$t_{1} - t_{2}$, we obtain
\begin{equation*}
\begin{split}
    f_{1}(t_{1},t_{2},s_{i};M) & = p_{3}(s_{i};M) - (t_{1} + t_{2})p_{2}(s_{i};M) + t_{1}t_{2}p_{1}(s_{i};M),\\
    f_{2}(t_{1},t_{2},s_{i};M) & = p_{4}(s_{i};M) - (t_{1}^{2} + t_{1}t_{2} + t_{2}^{2})p_{2}(s_{i};M) + t_{1}t_{2}(t_{1} + t_{2})p_{1}(s_{i};M),
\end{split} \quad i = 1,2.
\end{equation*}
In totality, this yields a parametrized system of four polynomial equations in the four unknowns $(t_{1},t_{2},s_{1},s_{2}) \in \CC^{4}$ with parameters being the entries of $M\colon$
\begin{equation}\label{eq:bisecant-system}
\begin{aligned}
F(t_{1},t_{2},s_{1},s_{2};M)
&=
\left[
\begin{aligned}
f_{1}(t_{1},t_{2},s_{1};M)\\
f_{1}(t_{1},t_{2},s_{2};M)\\
f_{2}(t_{1},t_{2},s_{1};M)\\
f_{2}(t_{1},t_{2},s_{2};M)
\end{aligned}
\right]\\[5pt]
&=
\left[
\begin{aligned}
&p_{3}(s_{1};M)
  - (t_{1}+t_{2})p_{2}(s_{1};M)
  + t_{1}t_{2}p_{1}(s_{1};M)\\[3pt]
&p_{3}(s_{2};M)
  - (t_{1}+t_{2})p_{2}(s_{2};M)
  + t_{1}t_{2}p_{1}(s_{2};M)\\[3pt]
&p_{4}(s_{1};M)
  - (t_{1}^{2}+t_{1}t_{2}+t_{2}^{2})p_{2}(s_{1};M)
  + t_{1}t_{2}(t_{1}+t_{2})p_{1}(s_{1};M)\\[3pt]
&p_{4}(s_{2};M)
  - (t_{1}^{2}+t_{1}t_{2}+t_{2}^{2})p_{2}(s_{2};M)
  + t_{1}t_{2}(t_{1}+t_{2})p_{1}(s_{2};M)
\end{aligned}
\right]
\end{aligned}
\end{equation}
where we only consider solutions
with $t_1\neq t_2$ and $s_1\neq s_2$.
Since the labeling of $t_1$ and~$t_2$
as well as the labeling of $s_1$ and $s_2$ 
are both arbitrary,
there is a natural $(S_2\times S_2)$-action
on the set of solutions
where the first $S_2$ permutes $t_1$ and~$t_2$
while the second $S_2$ permutes $s_1$ and $s_2$.
Hence, a common secant line corresponds to a single orbit of size four in the open subset $t_{1}\neq t_{2}, s_{1}\neq s_{2}$ of the solution set, namely
\begin{equation}\label{eq:S2xS2-action}
    (t_{1},t_{2},s_{1},s_{2}), \quad  (t_{2},t_{1},s_{1},s_{2}), \quad
    (t_{1},t_{2},s_{2},s_{1}), \text{ and} \quad (t_{2},t_{1},s_{2},s_{1}).
\end{equation}

\begin{proposition}\label{prop:40-lines}
For a general $[M] \in \mathrm{PGL}_{4}(\CC)$, the system~\eqref{eq:bisecant-system} has 40 isolated nonsingular solutions $(t_1,t_2,s_1,s_2)\in \CC^{4}$
and all such nonsingular solutions
satisfy $t_1\neq t_2$ and $s_1\neq s_2$.
These nonsingular solutions are 
grouped into ten orbits of size four under the $(S_{2}\times S_{2})$-action
in accordance with Theorem~\ref{thm:10-complex-secant-lines}. 
\end{proposition}
\begin{proof}
From  the structure of $F$
in \eqref{eq:bisecant-system},
if $s_1=s_2$, then $F$ only consists 
of two distinct polynomials in
three variables.
Hence, any solution
with $s_1=s_2$ must be lie
on a positive-dimensional component
and thus is singular.
Similarly, it is easy to observe
that $\partial F/\partial t_1(t_1,t_1,s_1,s_2;M) =  
\partial F/\partial t_2(t_1,t_1,s_1,s_2;M)$
so that any solution
with $t_1=t_2$ must also be~singular.

For the standard twisted cubic $C_0$    
and a general $C_1\in\mathscr{H}$,
the $10$ common secant lines 
are not contained in the plane
defined by $x_0=0$ in $\PP^3(\CC)$.
Hence, for general $[M]\in\mathrm{PGL}_{4}(\CC)$, there are $10$ common secant
lines between
$\{[v(a)] \mid a\in \CC\}$
and 
$\bigl\{[Mv(s)] \in \PP^{3}(\CC) \mid s\in \CC\bigr\}$.
Due to \eqref{eq:S2xS2-action},
these correspond with $40$
solutions of $F(t_1,t_2,s_1,s_2;M) = 0$.
with $t_1\neq t_2$ and $s_1\neq s_2$.
To show generic nonsingularity, it suffices to show nonsingularity
for a chosen $M$.  
For example, using the software system {\tt Macaulay2}~\cite{Macaulay2} developed by Grayson and Stillman, with the matrix
$$M =     \begin{pmatrix}
    1   &  0     &0  &   0 \\
     0  &   2   &  3  &   3\\
     0  &   1  &   1   &  2\\
     0  &   1 &    3    & 1
    \end{pmatrix},$$
the ideal
$J = \langle F\rangle : \langle (t_1-t_2)(s_1-s_2) \rangle^\infty \subset \QQ[t_1,t_2,s_1,s_2]$
has $\deg J = \deg \sqrt{J} = 40$ showing that all $40$ solutions
satisfying $t_1\neq t_2$ and $s_1\neq s_2$
are nonsingular.
\end{proof}

For parameterized polynomial systems such as $F$ in \eqref{eq:bisecant-system}, the monodromy group associated with the generically nonsingular isolated solutions is an invariant that encodes information about the structure of the solutions (e.g., see~\cite{MonodromyGroupCompute} by Hauenstein, Rodriguez, and Sottile). Due to the relationship described in \eqref{eq:S2xS2-action}, Theorem~\ref{thm:Monodromy} posits that the monodromy group associated with the 40 generically nonsingular isolated solutions of $F$ is $S_{10}\times S_2 \times S_2$. We will prove Theorem~\ref{thm:Monodromy} in Section~\ref{sec:MonodromyGroup}.

Now, suppose that the entries of $M$ are real so that $C_{1}$ is a real twisted cubic curve.
Given $(t_{1},t_{2},s_{1},s_{2})\in\CC^{4}$ with
$t_{1}\neq t_{2}$, $s_{1}\neq s_{2}$, and
$F(t_{1},t_{2},s_{1},s_{2};M)=0$, 
we now need to determine if the corresponding common secant line $\ell$ $\vcentcolon = \ell(t_{1},t_{2})$ 
in \eqref{eq:projective-secant-C0}
is totally real, partially real, minimally real, or nonreal.  The following describes this process.

First, one determines if $\ell$ is a real line. A real secant line to $C_0$ must intersect $C_0$ in either two real points, implying $(t_{1},t_{2})\in\RR^2,$ or a pair of nonreal complex conjugate points, implying $t_{2}=\conj(t_{1})$ where $\conj(a)$ is the complex conjugate of $a$. 
Conversely, if $(t_{1},t_{2})\in\RR^2$, then $\ell$ is trivially observed to be a real line and, if $t_{2}=\conj(t_{1})$, then $\ell$ is a line that interpolates between two nonreal complex conjugate points and is therefore real.

Similarly, $\ell$ is real if and only if the intersection points of $\ell$ with $C_{1}$, i.e. $Mv(s_{i})$, $i=1,2$, are either real or nonreal complex conjugates. Since $M$ is real and invertible, this is equivalent to $(s_{1},s_{2})\in\RR^{2}$ or $s_{2}=\conj(s_{1})$, respectively. This yields the following approach to classifying common secant lines.

\begin{proposition}\label{prop:classify}
If the entries of $M$ are real and $(t_{1},t_{2},s_{1},s_{2})\in\CC^{4}$ is a solution to $F(t_{1},t_{2},s_{1},s_{2};M) = 0$ such that $t_{1}\neq t_{2},s_{1}\neq s_{2}$, then the corresponding common secant line $\ell(t_{1},t_{2})$ to $C_0$ and $C_{1}$ as in~\eqref{eq:projective-secant-C0} is classified as follows$\colon$
\begin{itemize}
    \item $\ell$ is \defn{\emph{totally real}} if and only if $(t_{1},t_{2},s_{1},s_{2})\in\RR^{4}$; 
    \item $\ell$ is \defn{\emph{partially real}} if and only if 
    either $(t_{1},t_{2})\in\RR^{2}$
    and $(s_{1},s_{2}) \in \CC^{2}\setminus \RR^{2}$ such that $s_{2}=\conj(s_{1}),$ or $(s_{1},s_{2})\in\RR^{2}$ and $(t_{1},t_{2})\in \CC^{2}\setminus \RR^{2}$
    such that $t_{2}=\conj(t_{1})$;
    \item $\ell$ is \defn{\emph{minimally real}} if and only if $(t_{1},t_{2},s_{1},s_{2}) \in \CC^{4}\setminus \RR^{4}$
    such that $t_{2}=\conj(t_{1})$ and $s_{2}=\conj(s_{1})$;
    \item $\ell$ is \defn{\emph{nonreal}} if and only if 
    $(t_{1},t_{2})\in\CC^{2}\setminus\RR^{2}$ and $t_{2}\neq \conj(t_{1})$.
\end{itemize}
\end{proposition}

\begin{example}[10 totally real lines]\label{ex:10TotallyReal}
Take the following parameters as exact:
\begin{equation}\label{eq:parameter-p0}
M  = 
\begin{pmatrix}
1 & 0 & 0 & 0 \\
\phantom{-}1.4351 & \phantom{-}0.6797 & \phantom{1}3.9435 & \phantom{-}7.7238 \\
-1.3085 & \phantom{-}4.6694 & 13.1949 & \phantom{-}5.1509 \\
\phantom{-}1.1573 & \phantom{-}1.2007 & \phantom{1}5.7272 & \phantom{-}8.6591
\end{pmatrix}.
\end{equation}
Then, there are 40 isolated solutions in $\CC^4$ to $F(t_{1},t_{2},s_{1},s_{2};M) = 0$ 
with $t_1\neq t_2$ and $s_1\neq s_2$
which arise in 10 groups of four as in~\eqref{eq:S2xS2-action}. We utilized \texttt{alphaCertified} 
to certify that all 40 are real.  Rounded to four decimal places, one representative from each group is:
\begin{figure}[htbp]
\begin{picture}(220,150)
\put(-40,70){$
\begin{array}{lrrrrrrr} 
x_{1} &\vcentcolon=& (& 0.9750,& 0.1977,& 0.4144,& 2.4773 &),\\
x_{2} &\vcentcolon=& (& 0.1052,& 1.0123,& 0.4063,& -0.7032 &),\\
x_{3} &\vcentcolon=& (& 0.1089,& 1.0266,& 2.7003,& -0.7031 &),\\
x_{4} &\vcentcolon=& (& 1.0006,& -1.0064,& 0.2684,& -0.8378 &),\\
x_{5} &\vcentcolon=& (& 1.0311,& -1.0077,& -0.8386,& -2.0733 &),\\
x_{6} &\vcentcolon=& (& 0.9846,& -1.0399,& -1.8183,& 0.2619 &),\\
x_{7} &\vcentcolon=& (& -0.0610,& -0.8453,& -0.4690,& -0.0085 &),\\
x_{8} &\vcentcolon=& (& -1.0078,& 0.1115,& -0.8657,& -0.5436 &),\\
x_{9} &\vcentcolon=& (& -0.6075,& -0.3172,& -0.4000,& -0.0465 &),\\
x_{10} &\vcentcolon=& (& -1.0093,& 0.1324,& -0.8664,& 0.0316 &).\\
\end{array}
$}
\end{picture}
\end{figure}

\noindent Each of these corresponds to a totally real common real secant line to $C_0$ and~$C_{1},$ as illustrated in Figure~\ref{fig:10000}. 
\begin{figure}[htbp]
    \centering
    \begin{picture}(300,140)
    \put(60,0){\includegraphics[scale=0.3]{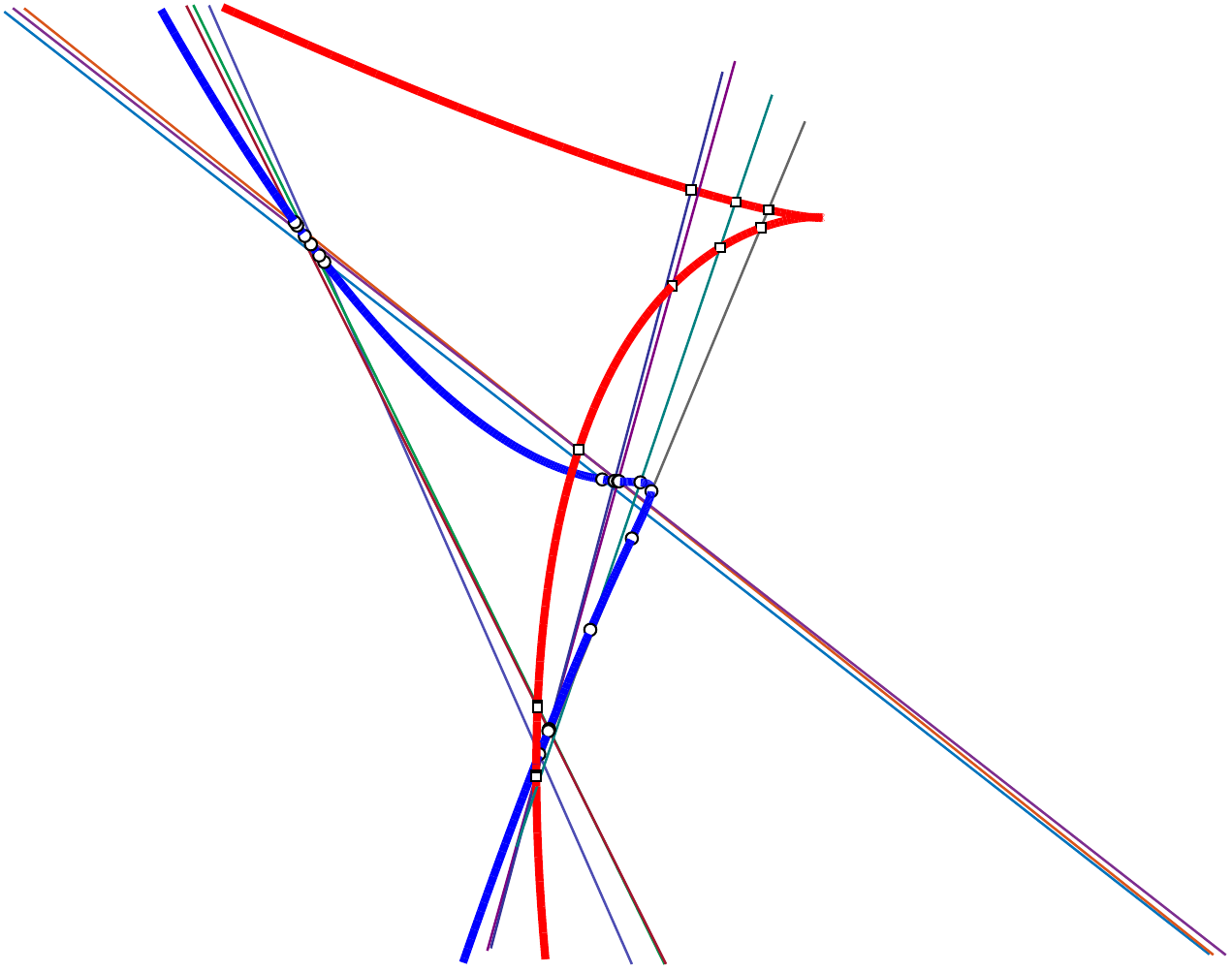}}
    \put(110,5){$C_{0}$}
    \put(120,135){$C_{1}$}
    \put(-60,0){\includegraphics[scale=0.25]{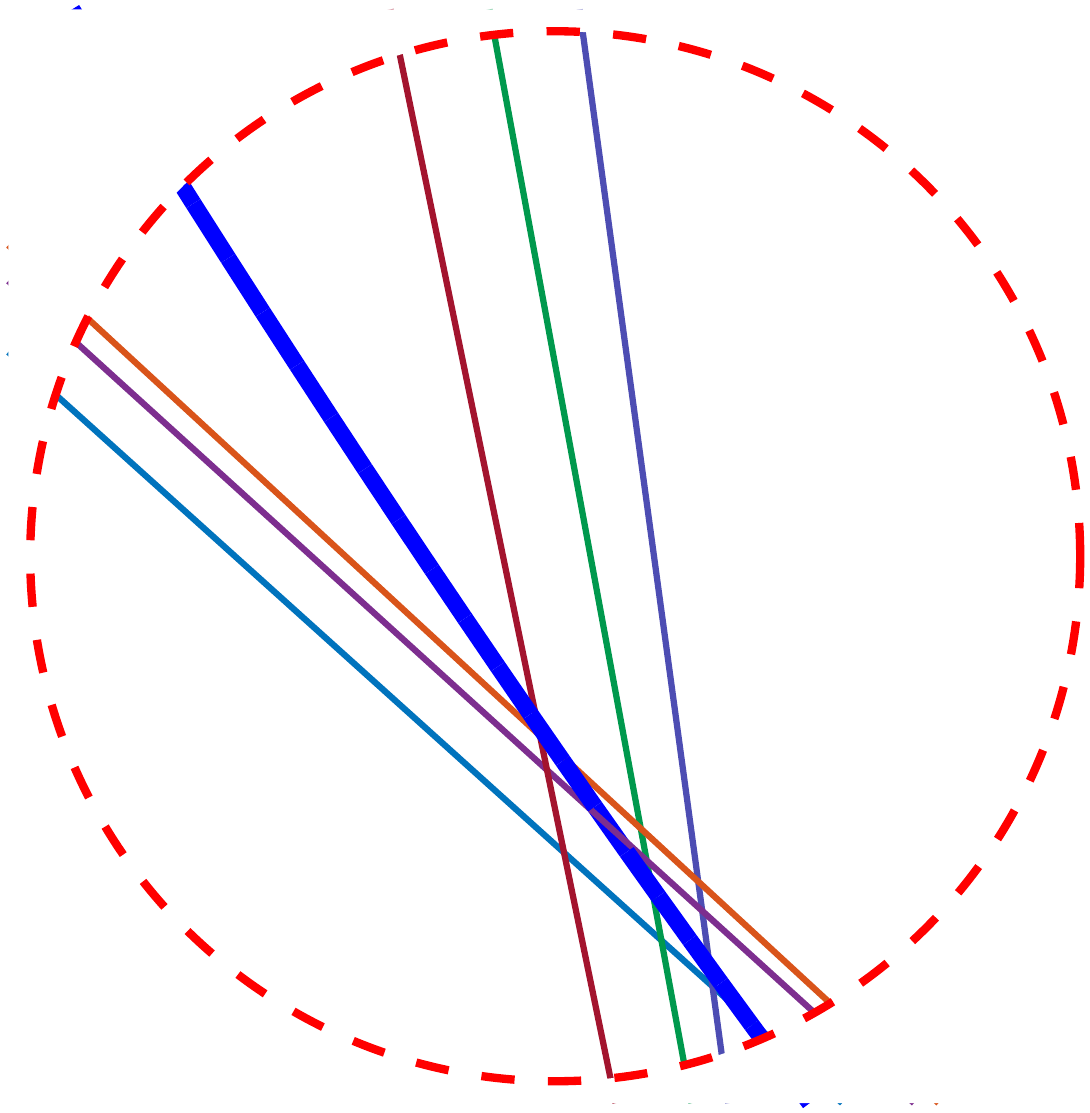}}
    \put(63,100){\tikzmark{base-blue}}
    \put(95,110){\tikzmark{arrowhead-blue}}
    \put(-45,60){$\ell_{1}$}
    \put(-65,95){$\ell_{2}$}
    \put(-40,90){$\ell_{3}$}
    \put(-20,110){$\ell_{4}$}
    \put(-6,105){$\ell_{5}$}
    \put(15,110){$\ell_{6}$}
    \put(240,10){\includegraphics[scale=0.25]
    {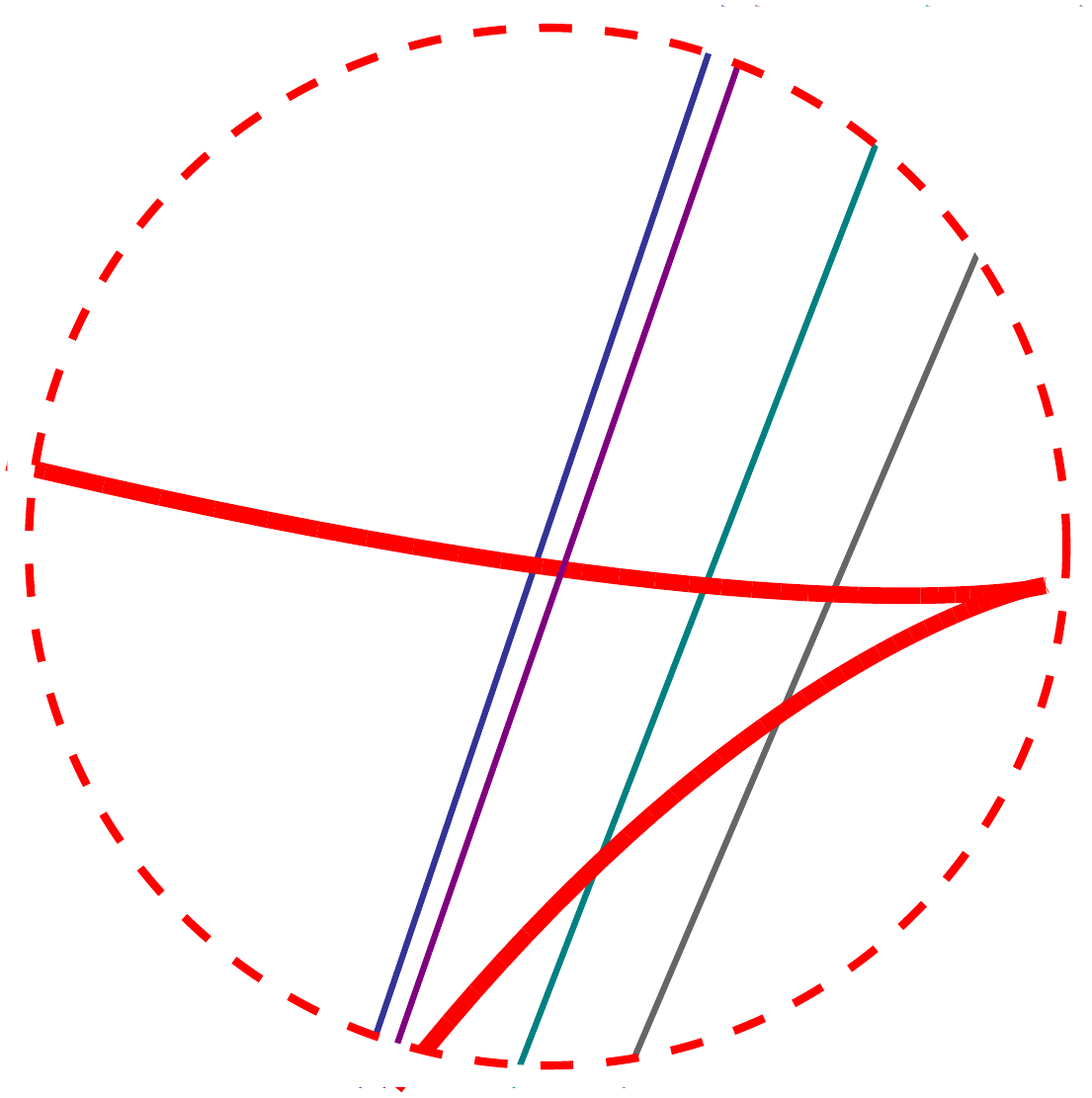}}
    \put(245,100){\tikzmark{base-red}}
    \put(185,115){\tikzmark{arrowhead-red}}
    \put(310,127){$\ell_{7}$}
    \put(328,115){$\ell_{8}$}
    \put(342,105){$\ell_{9}$}
    \put(353,87){$\ell_{10}$}
    \end{picture}
    \begin{tikzpicture}[remember picture, overlay]
    \draw [->, line width=1pt] ({pic cs:base-blue}) to ({pic cs:arrowhead-blue});
    \draw [->, line width=1pt] ({pic cs:base-red}) to ({pic cs:arrowhead-red});
    \end{tikzpicture}
    \caption{Ten totally real common secant lines $\ell_{1},\dots,\ell_{10}$ of $C_0$ and $C_{1}$.}
    \label{fig:10000}
\end{figure}

\noindent The $3$-tuple $(n_t,n_p,n_m)=(10,0,0)$ is realized by these two curves.
\end{example}

We can build off of the description in Proposition~\ref{prop:classify} to determine which of the admissible $3$-tuples $(n_t,n_p,n_m)$ are realized 
by taking into account the behavior of solutions to systems of polynomial equations with real coefficients, as common nonreal secant lines must arise in complex conjugate pairs.

\begin{proposition}\label{prop:admissible-s}
Let $\NN_0 = \{0,1,2,\dots\}$ be the set of nonnegative integers. The total number of admissible $3$-tuples is
$$
\#\left(\bigcup_{n_{\RR}\in\{0,2,4,6,8,10\}} \{(n_t,n_p,n_m)\in\NN_0^3\suchthat
n_t+n_p+n_m = n_{\RR}\}\right) = \sum\limits_{n_{\RR}\in \{0,2,4,6,8,10\}} \binom{n_{\RR}+2}{2} = 161.$$
\end{proposition}
\begin{proof}
To find the total number of admissible $3$-tuples, 
we use a version of the Stars and Bars Theorem from combinatorics (e.g., see~\cite[\S 1.2]{Stanley} by Stanley) which posits that 
there are exactly $\binom{m+2}{2}$
$3$-tuples of nonnegative integers whose sum is $m$. Since $n_{\RR}$ must be even, adding the $6$ values corresponding to even sums 
from $0$ to $10$ yields the result.
\end{proof}

\section{Monodromy}\label{sec:MonodromyGroup}

The monodromy group of ten common secant lines to twisted cubic curves is a subgroup of $S_{10}$.  This monodromy group encodes 
which permutations of
the ten common secant lines are possible 
as the pair of twisted cubic curves move in a closed loop such that the common secant lines 
vary smoothly.
We prove this group is $S_{10}$, i.e. Theorem~\ref{thm:Monodromy}, by finding two closed loops for which the corresponding monodromy action generates $S_{10}$ using $F$ in \eqref{eq:bisecant-system}. To simplify notation, set $x\vcentcolon=(t_1,t_2,s_1,s_2)\in\CC^4$ so that \eqref{eq:bisecant-system}
can be written as $F(x;M)=0$.

We utilize the following procedure,
which has been 
utilized for other problems before (see, e.g., ~\cite{GaloisSchubert} by Hauenstein, Haywood, and Liddell, and~\cite{AposterioCertification} by Leykin and Sottile).
We select a sufficiently general 
base point $[M_{0}]\in\mathrm{PGL}_{4}(\CC)$, for which $F(x;M_{0}) = 0$ has~$40$ nonsingular isolated solutions
which arise in 10 groups of four, as summarized in Proposition~\ref{prop:40-lines}. 
This is sufficient since~$\mathscr{H}$ is path-connected. For each of the 10 groups, we select a representative and assign an ordering $x_1,\dots,x_{10}$. 
We then follow a two-step process to prove Theorem~\ref{thm:Monodromy}$\colon$ a heuristic search followed by certification. 

For the proof of our result, we will take $M_0$
to be the parameters listed in Example~\ref{ex:10TotallyReal} 
and 
$x_1,\dots,x_{10}$ to be the points 
as listed there as well.

\subsection{Heuristic searching for triangular loops}

Suppose that $\gamma:[0,1]\rightarrow\mathrm{PGL}_{4}(\CC)$ is a loop with $\gamma(0) = \gamma(1) = [M_0]$ 
such that $F(x;\gamma(a)) = 0$
has $40$ nonsingular isolated solutions 
for every $a\in [0,1]$. Then, for every $i=1,\dots,10$, there is a path $X_i:[0,1]\rightarrow\CC^4$ such that~$X_i(a)$ is a nonsingular solution of $F(x;\gamma(a))=0$ for every $a\in [0,1]$ with $X_i(0) = x_i$. Since~$\gamma$ is a loop, there exists a
unique $j_i\in\{1,\dots,10\}$ such that~$X_i(1)$ 
is in the same $(S_2\times S_2)$-orbit
of~$x_{j_i}$.  Let $\sigma_\gamma\in S_{10}$ 
be the permutation this describes, i.e.,
$\sigma_\gamma(i) = j_i$.
The monodromy group consists of all
such permutations $\sigma_\gamma$.

For our heuristic search, we generate loops $\gamma$ as the edges of a triangle 
where one of the vertices is $M_{0},$ the aforementioned fixed base point.
We select the other two vertices randomly,
and track the corresponding solution paths $X_i(a)$ using numerical path tracking implemented in {\tt Bertini} to determine what we purport to be the permutation $\sigma_\gamma$.

Since $S_{10}$ can be generated by two permutations, we heuristically searched for two loops~$\gamma_1$ and $\gamma_2$ such that the purported permutations $\sigma_{\gamma_1}$ and $\sigma_{\gamma_2}$ generate $S_{10}$. 
Each edge of the triangle loop $\gamma_{k}$ is parametrized by a straight-line homotopy in the parameter space, as illustrated in 
Figure~\ref{fig:two-loops}.

\smallskip
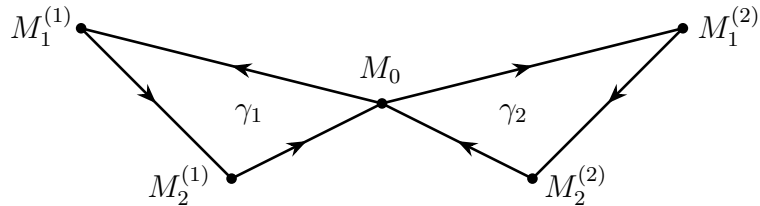
\begin{figure}[htbp]
\begin{tikzpicture}[line cap = round, line join = round]
\tikzset{
  midarrow/.style={
    line width=1pt,
    postaction={
      decorate,
      decoration={
        markings,
        mark=at position 0.5 with {\arrow{Stealth}}
      }
    }
  }
}
    
    \coordinate (A) at (8,2);
    \coordinate (B) at (6,3);
    \coordinate (C) at (10,4);
    \coordinate (D) at (2,4);
    \coordinate (E) at (4,2);

    \draw[midarrow] (B) -- (C);
    \draw[midarrow] (C) -- (A);
    \draw[midarrow] (A) -- (B);

    \draw[midarrow] (B) -- (D);
    \draw[midarrow] (D) -- (E);
    \draw[midarrow] (E) -- (B);
    
    \fill[black] (A) circle(2pt);
    \fill[black] (B) circle(2pt);
    \fill[black] (C) circle(2pt);
    \fill[black] (D) circle(2pt);
    \fill[black] (E) circle(2pt);

    \put(162,95){$M_{0}$}
    
    \put(115,80){$\gamma_{1}$}
    \put(30,110){$M_{1}^{(1)}$}
    \put(82,50){$M_{2}^{(1)}$}

    \put(215,80){$\gamma_{2}$}
    \put(290,110){$M_{1}^{(2)}$}
    \put(232,50){$M_{2}^{(2)}$}
\end{tikzpicture}
\caption{Illustration of loops $\gamma_{1}$ and $\gamma_{2}$.}
\label{fig:two-loops}
\end{figure}
If an edge joins $M_{\alpha}^{(k)}$ to $M_{\beta}^{(k)}$, we track
along the straight line 
$aM_{\alpha}^{(k)} + (1-a)M_{\beta}^{(k)}$ for $a\in [0,1]$.
In order for the path to be well-defined
in $\mathrm{PGL}_{4}(\CC)$, 
we need to ensure that 
$\det\bigl(aM_{i}^{(k)} + (1-a)M_{j}^{(k)}\bigr)\neq 0$ for all $a\in [0,1]$.
Since $\det\bigl(aM_{i}^{(k)} + (1-a)M_{j}^{(k)}\bigr)$ is a polynomial in $a$, 
one just needs to check that the roots do not lie in $[0,1]$.  
Tables~\ref{tab:roots-loop1} 
and~\ref{tab:roots-loop2}
show the 
roots for the corresponding path $\gamma_1$
and $\gamma_2$, respectively.

\begin{table}[ht]
\centering
\small
\setlength{\tabcolsep}{6pt}
\renewcommand{\arraystretch}{1.2}
\begin{tabular}{c c l}
\toprule
Edge & Determinant & Roots of $\det = 0$ \\
\midrule
$M_0$--$M_1^{(1)}$
& $\det\big(aM_0 + (1-a)M_1^{(1)}\big)$
&
$\begin{aligned}
a_{1} &\approx \phantom{-}0.9942 + 0.0016\mathrm{i}\\
a_{2} &\approx \phantom{-}1.0449 + 0.0260\mathrm{i}\\
a_{3} &\approx \phantom{-}1.3711 + 0.0331\mathrm{i}\\
a_{4} &\approx \phantom{-}2.6126 + 0.7410\mathrm{i}
\end{aligned}$
\\[8ex]

$M_1^{(1)}$--$M_2^{(1)}$
& $\det\big(aM_1^{(1)} + (1-a)M_2^{(1)}\big)$
&
$\begin{aligned}
a_{1} &\approx -0.5899 + 1.5925\mathrm{i}\\
a_{2} &\approx \phantom{-}0.0428 - 0.4212\mathrm{i}\\
a_{3} &\approx \phantom{-}0.3658 + 0.3845\mathrm{i}\\
a_{4} &\approx \phantom{-}0.4319 - 0.0771\mathrm{i}
\end{aligned}$
\\[8ex]

$M_2^{(1)}$--$M_0$
& $\det\big(aM_2^{(1)} + (1-a)M_0\big)$
&
$\begin{aligned}
a_{1} &\approx -0.0044 + 0.0110\mathrm{i}\\
a_{2} &\approx \phantom{-}0.0214 - 0.5554\mathrm{i}\\
a_{3} &\approx \phantom{-}0.0348 - 0.0091\mathrm{i}\\
a_{4} &\approx \phantom{-}0.8297 + 0.7387\mathrm{i}
\end{aligned}$
\\
\bottomrule
\end{tabular}
\caption{Roots of determinant polynomials along the edges of loop~$\gamma_{1}$.}
\label{tab:roots-loop1}
\end{table}

\begin{table}[ht]
\centering
\small
\setlength{\tabcolsep}{6pt}
\renewcommand{\arraystretch}{1.2}
\begin{tabular}{c c l}
\toprule
Edge & Determinant & Roots of $\det = 0$ \\
\midrule
$M_0$--$M_1^{(2)}$
& $\det\big(aM_0 + (1-a)M_1^{(2)}\big)$
&
$\begin{aligned}
a_{1} &\approx \phantom{-}0.8047 + 0.4024\mathrm{i}\\
a_{2} &\approx \phantom{-}0.9982 - 0.0000\mathrm{i}\\
a_{3} &\approx \phantom{-}1.6446 - 0.6637\mathrm{i}
\end{aligned}$
\\[6ex]

$M_1^{(2)}$--$M_2^{(2)}$
& $\det\big(aM_1^{(1)} + (1-a)M_2^{(2)}\big)$
&
$\begin{aligned}
a_{1} &\approx -2.1843 + 4.3611\mathrm{i}\\
a_{2} &\approx \phantom{-}0.4278 + 0.2305\mathrm{i}\\
a_{3} &\approx \phantom{-}0.6036 - 1.0162\mathrm{i}
\end{aligned}$
\\[6ex]

$M_2^{(2)}$--$M_0$
& $\det\big(aM_2^{(2)} + (1-a)M_0\big)$
&
$\begin{aligned}
a_{1} &\approx -3.5780 + 0.1699\mathrm{i}\\
a_{2} &\approx -0.0684 - 0.0127\mathrm{i}\\
a_{3} &\approx -0.0099 + 0.0070\mathrm{i}
\end{aligned}$
\\
\bottomrule
\end{tabular}
\caption{Roots of determinant polynomials along the edges of loop~$\gamma_{2}$.}
\label{tab:roots-loop2}
\end{table}

\noindent The loops $\gamma_1$ and $\gamma_2$
are defined by the following matrices:
\begin{equation*}
M_{1}^{(1)}{=} 
\begin{pmatrix}
\phantom{-}4.1182{-}3.6621\mathrm{i} & -5.6621{+}6.1182\mathrm{i} & \phantom{-}3.1182{-}4.8843\mathrm{i} & -6.1182{+}5.6621\mathrm{i} \\
\phantom{-}1.1182{-}2.1182\mathrm{i} & -10.6621{-}3.1182\mathrm{i} & \phantom{-}5.1182{+}0.8843\mathrm{i} & \phantom{-}8.1182{-}0.8843\mathrm{i} \\
-1.1182{-}1.1182\mathrm{i} & \phantom{-}18.1182{+}0.8843\mathrm{i} & \phantom{-}25.1182{-}9.1182\mathrm{i} & \phantom{-}3.1182{+}0.8843\mathrm{i} \\
\phantom{-}12.1182{+}8.1182\mathrm{i} & \phantom{-}3.1182{-}0.8843\mathrm{i} & \phantom{-}8.1182{-}2.1182\mathrm{i} & \phantom{-}20.1182{-}6.1182\mathrm{i}
\end{pmatrix},
\end{equation*}

\begin{equation*}
M_{2}^{(1)}{=} 
\begin{pmatrix}
-1.8843{+}2.6621\mathrm{i} & \phantom{-}4.6621{-}7.1182\mathrm{i} & -5.5529{+}1.4412\mathrm{i} & \phantom{-}5.6621{-}2.1182\mathrm{i} \\
-7.8843{+}6.4412\mathrm{i} & \phantom{-}2.1182{-}7.8843\mathrm{i} & -10.6621{+}5.1182\mathrm{i} & \phantom{-}20.4412{+}4.6621\mathrm{i} \\
\phantom{-}7.5529{-}3.1182\mathrm{i} & -6.1182{-}4.6621\mathrm{i} & \phantom{-}10.6621{+}5.4412\mathrm{i} & -7.8843{+}9.1182\mathrm{i} \\
\phantom{-}2.4412{-}5.1182\mathrm{i} & \phantom{-}4.6621{-}10.6621\mathrm{i} & -3.1182{+}7.8843\mathrm{i} & -0.6621{-}6.4412\mathrm{i}
\end{pmatrix},
\end{equation*}

\begin{equation*}
M_{1}^{(2)}{=} 
\begin{pmatrix}
1 & 0 & 0 & 0 \\
-0.9132{+}13.7745\mathrm{i} & -14.8821{-}1.3348\mathrm{i} & \phantom{.}20.4413{+}2.1187\mathrm{i} & -6.7724{-}14.2031\mathrm{i} \\
\phantom{-}15.9982{-}3.1146\mathrm{i} & \phantom{.}5.4418{+}12.9926\mathrm{i} & \phantom{.}26.1184{+}2.7741\mathrm{i} & \phantom{.}3.8825{+}16.1183\mathrm{i} \\
-2.1187{+}10.4419\mathrm{i} & \phantom{.}15.9923{-}1.8824\mathrm{i} & -7.4412{+}0.9926\mathrm{i} & \phantom{.}9.1184{-}11.7745\mathrm{i}
\end{pmatrix}
,
\end{equation*}

\begin{equation*}
M_{2}^{(2)}{=}
\begin{pmatrix}
1 & 0 & 0 & 0 \\
-11.4413{+}1.1184\mathrm{i} & \phantom{-}14.8826{-}0.7742\mathrm{i} & -8.9921{+}11.4410\mathrm{i} & \phantom{.}6.1183{-}13.9926\mathrm{i} \\
-2.8824{-}15.4417\mathrm{i} & \phantom{.}4.1186{+}9.8821\mathrm{i} & \phantom{.}24.9928{+}3.1184\mathrm{i} & -6.1187{+}11.7742\mathrm{i} \\
-10.4412{+}13.8826\mathrm{i} & \phantom{.}12.1184{-}9.4413\mathrm{i} & \phantom{.}4.8821{+}1.1184\mathrm{i} & \phantom{.}19.7745{-}0.8826\mathrm{i}
\end{pmatrix}.
\end{equation*}

In our heuristic process, we found that the loops $\gamma_1$ and $\gamma_2$ appeared to yield
\begin{equation}\label{eq:Permutations}
    \sigma_{\gamma_{1}} = (1 \ 7 \ 5 \ 9 \ 8 \ 6 \ 3 \ 4 \ 2)(10)\quad \text{and} \quad \sigma_{\gamma_{2}} = (1 \ 3 \ 4 \ 7 )(2 \ 5)(6 \ 8 \ 10 \ 9).
\end{equation}
It remains to certify that the numerical path tracking indeed produces these permutations.

\subsection{Certification}

Although path tracking in {\tt Bertini} is robust, it is heuristic and thus one needs to utilize certification to prove that $\gamma_1$ and $\gamma_2$ yield 
smooth solution paths and 
the purported permutations~$\sigma_{\gamma_1}$ and $\sigma_{\gamma_2}$ in \eqref{eq:Permutations}, respectively.

\medskip
\begin{proof}[Proof of Theorem~\ref{thm:Monodromy}]

For $j=1,2$ and $i=1,\dots,10$, we utilized the certified homotopy tracking software \texttt{CertifiedHomotopyTracking.jl} arising from the Krawczyk homotopy method from Duff and Lee in~\cite{duff-lee}
to prove that the path $X_{i}^{(j)}(a)$ defined by $F(X_{i}^{(j)}(a);\gamma_j(a)) = 0$ for $a\in[0,1]$ with $X_{i}^{(j)}(0) = x_i$ has $X_{i}^{(j)}(a)$ being a nonsingular isolated solution of 
$F(x;\gamma_j(a))=0$ and $X_i^{(j)}(1) = x_{\sigma_{\gamma_j}(i)}$
where $\sigma_{\gamma_j}$ in \eqref{eq:Permutations}. 
All scripts, input files, and certification certificates are available at the public {\tt GitHub} repository~\cite{github-repo-AMS-MRC-2025}. Finally, we execute the following code in {\tt Mathematica}:
\begin{leftbar}
\begin{verbatim}
    s1 = Cycles[{{1,7,5,9,8,6,3,4,2},{10}}];
    s2 = Cycles[{{1,3,4,7},{2,5},{6,8,10,9}}];
    GroupOrder[PermutationGroup[{s1,s2}]]
\end{verbatim}
\end{leftbar}
\noindent
The output is 3,628,800. Thus, $\langle \sigma_{\gamma_1},\sigma_{\gamma_2}\rangle$ is a subgroup of $S_{10}$ of size 3,628,800 = 10!. Since $\# S_{10} = 10!$, this proves $\langle \sigma_{\gamma_1},\sigma_{\gamma_2}\rangle = S_{10}$.
\end{proof}

\section{Realizable 3-tuples}\label{sec:real-secants}

We sample the parameter space \defn{$\mathscr{H}(\RR)$} $\vcentcolon = \mathrm{PGL}_{4}(\RR)/\mathrm{PGL}_{2}(\RR)$ to investigate how the number of real and nonreal solutions to~\eqref{eq:bisecant-system} varies.

\subsection{Uniform sampling} To explore $\mathscr{H}(\RR)$, we sample uniformly in 
the real projective space~$\PP^{15}(\RR)$. Since $\PP^{15}(\RR)$ can be realized as the quotient $\mathbb{S}^{15}/\{\pm 1\}$, where \defn{$\mathbb{S}^{15}$} $\subset \RR^{16}$ is the unit sphere, with antipodal points identified, we first sample $\mathbb{S}^{15}$ uniformly and then select a representative from each antipodal pair. Uniform sampling on $\mathbb{S}^{15}$ is performed by drawing vectors from the standard Gaussian distribution in $\RR^{16}$ and normalizing them to unit length. To pass to projective space, we fix a representative for each line through the origin by requiring the last coordinate be nonnegative.

A \texttt{Python} driver script automates parameter sampling, tracking 
parameter homotopies in \texttt{Bertini}, and the 
certified classification of real versus complex solutions. Each sample produces a $3$-tuple $(n_{t},n_{p},n_{m})$ representing the number of various real secant lines. Algorithm~\ref{alg:uniform-sampling} summarizes this approach.

\begin{algorithm}[H]
\caption{Uniformly sampling parameter points in $\mathscr{H}(\RR)$ for the system $F$~\eqref{eq:bisecant-system}}\label{alg:uniform-sampling}
\begin{algorithmic}
\State \textbf{Input:} The number $N$ of parameters to be sampled.
\State \textbf{Output:} A list of 3-tuples $(n_{t},n_{p},n_{m})$ giving the count of real secant lines obtained from the $N$ parameters uniformly sampled.
\State 1. Generate $N$ sample parameter points $M_{1},\dots,M_N$ uniformly in $\PP^{15}(\RR)$.
\State \textbf{For each} $i \in  \{1,\dots, N\}$ \textbf{do}$\colon$
\State 2. Solve $F(x;M_{i})=0$. Generically there are 40 nonsingular isolated solutions.
\State 3. Identify real and complex-conjugate pairs among the solutions and pick 10 representatives $(t_{1},t_{2},s_{1},s_{2})\in \CC^{4}$ corresponding to their $(S_{2}{\times}S_{2})$-orbits.
\State 4. Classify each representative according to Proposition~\ref{prop:classify}.
\State \textbf{Return} A list of 3-tuples $(n_{t},n_{p},n_{m})$ giving the count of real secant lines.
\end{algorithmic}
\end{algorithm}

We performed 100,000 uniform random samples of $\mathscr{H}(\RR)$. Figure~\ref{fig:nonreal} shows that the distribution of real common secants is concentrated at $n_{\RR} = 2$ and $n_{\RR} = 4$. Configurations with ten real secant lines occurred only in 0.06\% of samples. 

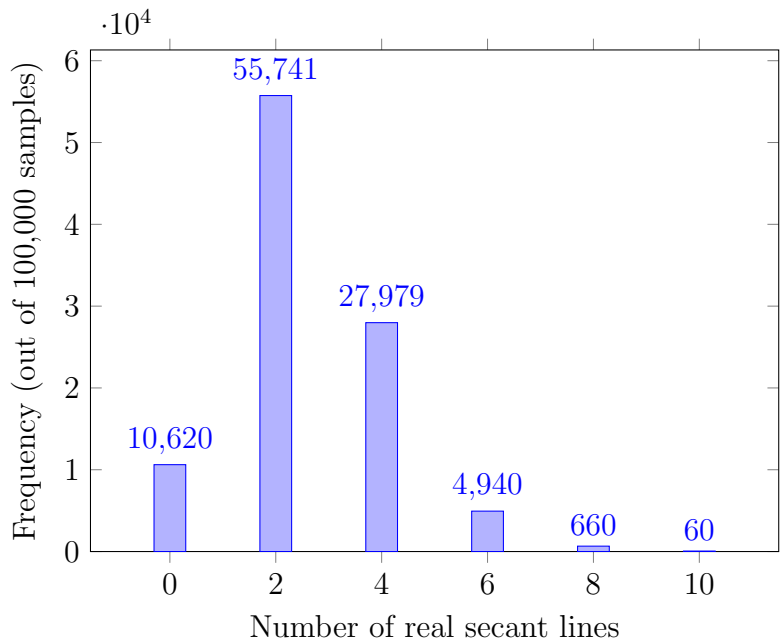
\begin{figure}[ht]
\begin{tikzpicture}
\begin{axis}[
    ybar,
    bar width=12pt,
    width=0.65\textwidth,
    height=0.5\textwidth,
    ymin=0,
    xlabel={Number of real secant lines},
    ylabel={Frequency (out of 100,000 samples)},
    xtick=data,
    enlarge x limits=0.15,
    nodes near coords,
    nodes near coords align={vertical},
]
\addplot coordinates {
(10,60)
(8,660)
(6,4940)
(4,27979)
(2,55741)
(0,10620)
};
\end{axis}
\end{tikzpicture}
\caption{Frequency distribution of the number of real common secant lines among 100,000 uniform random samples.}
\label{fig:nonreal}
\end{figure}

We observe that the induced distribution on 3-tuples $(n_{t},n_{p},n_{m})$ is similarly highly nonuniform, with a small number of 
3-tuples arising in the majority of samples.
Many admissible 3-tuples were not
realized in our set of 100,000 samples. 
Table~\ref{tab:-25} shows the 10 most frequent 3-tuples $(n_{t},n_{p},n_{m})$ realized by these 100,000 samples.

\begin{table}[htbp]
\centering
\caption{Top 10 most frequent 3-tuples $(n_{t},n_{p},n_{m})$ for 100{,}000 samples.}
\label{tab:-25}
\begin{tabular}{rrrrr}
\toprule
$n_{t}$ & $n_{p}$ & $n_{m}$ & Count \\
\midrule
0 & 1 & 1 & 21,830\\
0 & 2 & 0 & 11,811 \\
0 & 0 & 0 & 10,620\\
1 & 0 & 1 & 9,960\\
0 & 0 & 2 & 7,280 \\
0 & 4 & 0 & 5,309\\
0 & 3 & 1 & 4,340\\
1 & 3 & 0 & 3,510 \\
2 & 1 & 1 & 3,390\\
1 & 1 & 0 & 3,310 \\
\bottomrule
\end{tabular}
\end{table}

\subsection{Localizing sampling}

Rather than uniformly sampling,
we aim to find additional realizable 3-tuples
by generating samples near a given
parameter point. Algorithm~\ref{alg:real-secants} summarizes this localized approach,
which we use to 
complete the proof of our main theorem, Theorem~\ref{thm:main}.

\begin{algorithm}[htbp]
\caption{Sampling parameter points near a given point in $\mathscr{H}(\RR)$.}\label{alg:real-secants}
\begin{algorithmic}
\State \textbf{Input:} A sample point $[M_{0}] \in \mathscr{H}(\RR)$ representing parameters of $F(x;M_{0})=0$ as in~\eqref{eq:bisecant-system}, a number $\varepsilon \in (0,1)$, and a number $N$ of parameters to be sampled in the $\varepsilon$-neighborhood $B_{\varepsilon}(M_{0})$ of $M_{0}$.
\State \textbf{Output:}  A list of 3-tuples $(n_{t},n_{p},n_{m})$ giving the count of real secant lines obtained from the $N$ parameters sampled in $B_{\varepsilon}(M_{0})$.
\State 1. Generate $N$ sample parameter points 
$M_1,\dots,M_N$ in $B_{\epsilon}(M_{0})$.
\State \textbf{For each} $i\in N$ \textbf{do}:
\State 2. Run a parameter homotopy from the set of 10 solution representatives of $F(x;M_{0}) = 0$ to $F(x;M_{i})=0$.
\State 3. Classify each representative according to Proposition~\ref{prop:classify}.
\State \textbf{Return} A list of 3-tuples $(n_{t},n_{p},n_{m})$ giving the count of real secant lines.
\end{algorithmic}
\end{algorithm}

\medskip
\begin{proof}[Proof of Theorem~\ref{thm:main}]
By Proposition~\ref{prop:admissible-s}, there are a total of 161 admissible $3$-tuples $(n_{t},n_{p},n_{m})$. By sampling a large number of initial parameter points in $\mathscr{H}(\RR)$, and repeatedly performing Algorithm~\ref{alg:real-secants} on points in our sample whose 3-tuples $(n_{t},n_{p},n_{m})$ are located near admissible but not yet 
realized 3-tuples, we 
first used heuristic 
solving to determine
parameters $M \in \mathscr{H}(\RR)$ 
which realize an admissible 3-tuple
that had yet to be realized.
This process generated 
parameters for 
128 out of the 161 admissible 3-tuples.

For each of these parameters, we proved realizability by certifying approximate solutions using \texttt{alphaCertified}. The certified data, including rational representations of the solutions, are available at the {\tt GitHub} repository~\cite{github-repo-AMS-MRC-2025}; the parameters are archived at~\cite{CurateNDSecantData}.

Table~\ref{tab:3-tuples} in the Appendix shows the lists of realized 3-tuples 
as well as admissible 3-tuples that have not been realized, which is also plotted
in~Figure~\ref{fig:3-tuples-plot}.
An inspection of the data confirms that all admissible 3-tuples satisfying the hypothesis of the theorem have been realized and certified.
\end{proof}

\subsection{Completing the proof of Theorem 1.5}\label{subsec:totally-real}

Since Example~\ref{ex:10TotallyReal} showed the existence of two real twisted cubics with $10$ totally real secant lines, it remains to exhibit $10$ pairs of cubics with each one of the other possible integers from $0$ to $9$ as the number of totally real common secants. To this end, Table~\ref{tab:totally-real-examples} in Appendix~\ref{app:totally-real} lists parameters which produce these pairs; this completes our proof of Theorem~\ref{thm:totally-real}. In each example, the matrix $M$ has first row $(1,0,0,0)$, and $\widetilde{M}$ denotes the submatrix consisting of its last three rows. \qed

\section{Discussion and Open Problems}\label{sec:discussion-and-open-problems}

Over the complex numbers, Theorem~\ref{thm:10-complex-secant-lines} provides the well-known count of ten common secants to two general twisted cubic curves. Using computational algebraic geometry, we showed that the monodromy group of this enumerative problem is the full symmetric group $S_{10}$ (Theorem~\ref{thm:Monodromy}). In addition,
we formulated and investigated
the real analogue of Theorem~\ref{thm:10-complex-secant-lines} to 
determine the possible collections of ten common secant lines to a pair of real twisted cubics.
To accomplish this, we classified 
real secants as totally real, partially real, or minimally real, and used combinatorics
to determine that there are a total of 161 admissible 3-tuples $(n_{t},n_{p},n_{m})$. We then used \texttt{Python} scripts with \texttt{Bertini} to explore the parameter space $\mathscr{H}(\RR)$ to realize 
128 of the 161 admissible 3-tuples $(n_{t},n_{p},n_{m})$. We used \texttt{alphaCertified} 
to certify the solutions associated with the realizable 3-tuples found, yielding Theorem~\ref{thm:main}.
A particular case is formulated in
Theorem~\ref{thm:totally-real},
which posits that any number from $0$ to $10$
of totally real secant lines is possible.

We outline a few open problems naturally arising from our investigation.

\begin{problem}
Determine the full set of realizable 3-tuples in Table~\ref{tab:3-tuples}. Our computations show that 128 out of 161 admissible 3-tuples have been realized. 
Of the remaining 33 unresolved cases, 31 correspond to configurations in which all 10 common secants are real. As illustrated in Figure~\ref{fig:3-tuples-plot}, the 3-tuples $(1,0,7)$, $(1,1,6)$, and $(4,0,6)$ are adjacent to several realized 3-tuples. We expect these 3-tuples to be realized.
\end{problem}

\begin{problem} 
A generalization of Problem~\ref{problem:real} is to consider common $(n-1)$-secant $(n-2)$-planes to two rational normal curves in $\mathbb{P}^{n}$. For example, when $n = 4$, a real trisecant plane may intersect a rational normal quartic in either one or three distinct real points (see~\cite{Bernardi-etal} by Bernardi, Gimigliano, and Id\`a), giving rise to a classification of real common trisecant planes analogous to that of real common secants to twisted cubics.
\end{problem}

\section{Acknowledgments}

We thank the American Mathematical Society for its support, and the organizers of the 2025 AMS Mathematics Research Communities (MRC) on Real Numerical Algebraic Geometry for setting up this collaboration. We are also grateful to the mentors and participants of the MRC for valuable discussions throughout this project. In addition, we would like to show our appreciation for the anonymous referees, for their careful reading and suggestions. 

MF was partially supported by NSF DMS–2052519.
JDH was partially supported by NSF CCF-2331400
and the Robert and Sara Lumpkins Collegiate Professorship.
JLG was partially supported by the Robert and Sara Lumpkins Collegiate Professorship. 
CWW was partially supported by the Huisking Foundation, Inc., Collegiate Professorship.

During the preparation of this work, the authors used ChatGPT in order to create \texttt{shell} and \texttt{Python} scripts to post-process data obtained from \texttt{Bertini} and \texttt{alphaCertified}. After using these tools, the authors reviewed the results and utilized them in the work on an as-needed basis. The authors take full responsibility for the content of the published article.

\FloatBarrier

\bibliographystyle{abbrv}
\bibliography{paper}

\clearpage

\appendix
\counterwithin{table}{section}

\section{Parameters realizing totally real secant lines}
\label{app:totally-real}

\newcommand{\Msub}[3]{%
  \left(\begin{smallmatrix}
    #1\\#2\\#3
  \end{smallmatrix}\right)}

\newcommand{\SecantImage}[1]{%
  \raisebox{-0.5\height}{%
    \includegraphics[
      width=0.21\textwidth,
      height=1.9cm,
      keepaspectratio
    ]{#1.pdf}%
  }%
}

\begingroup
\small
\setlength{\tabcolsep}{5pt}
\renewcommand{\arraystretch}{1.15}

\begin{longtable}{@{}cccc@{}}
\caption{Parameters realizing 0 to 9 totally real common secant lines, together
with their corresponding graphs. These parameters are archived at~\cite{CurateNDSecantData}.}
\label{tab:totally-real-examples}\\

\toprule
\# of totally real lines & 3-tuple & submatrix $\widetilde M$
& graph \\
\midrule

$9$ & $(9,1,0)$ &
$\Msub
 { 0.7519 & 0.1524 &  2.8796 &  7.6158}
 {-1.0069 & 3.7048 & 15.1121 &  7.9388}
 { 1.2271 & 0.1795 &  3.6345 & 10.6654}$
&
\SecantImage{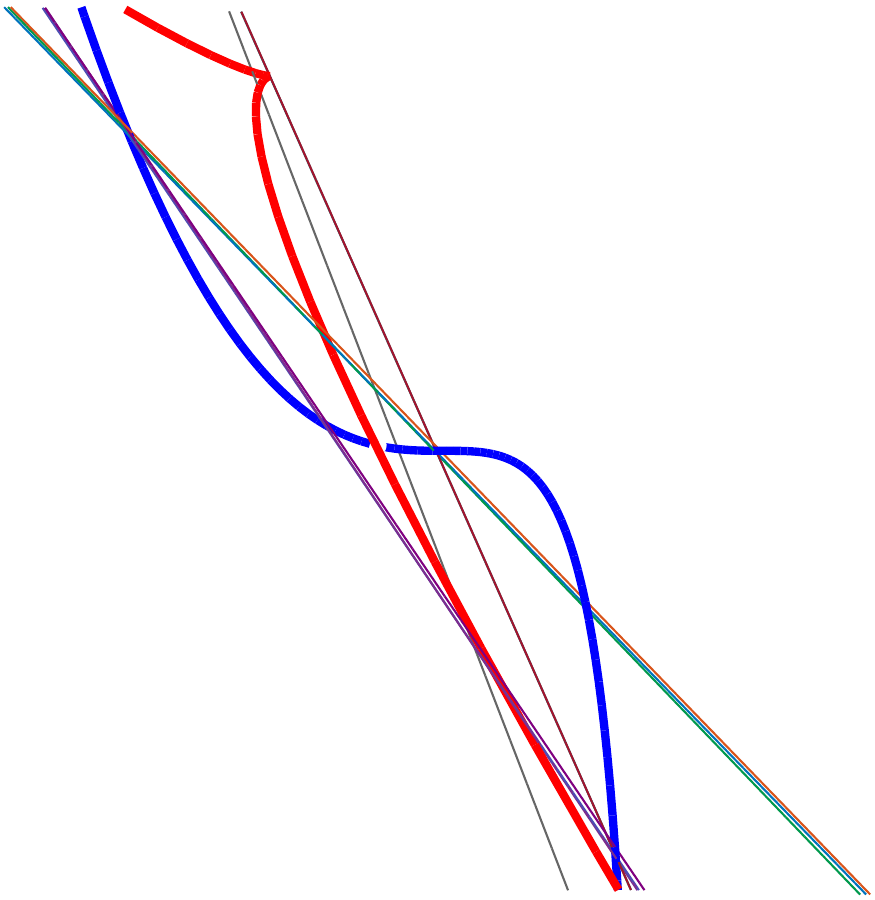}
\\[1ex]

$8$ & $(8,0,0)$ &
$\Msub
 { 0.7712 & 1.0684 &  4.0153 & 6.8183}
 {-0.3876 & 3.8568 & 13.8269 & 6.0562}
 { 0.7904 & 1.0239 &  4.9333 & 9.2222}$
&
\SecantImage{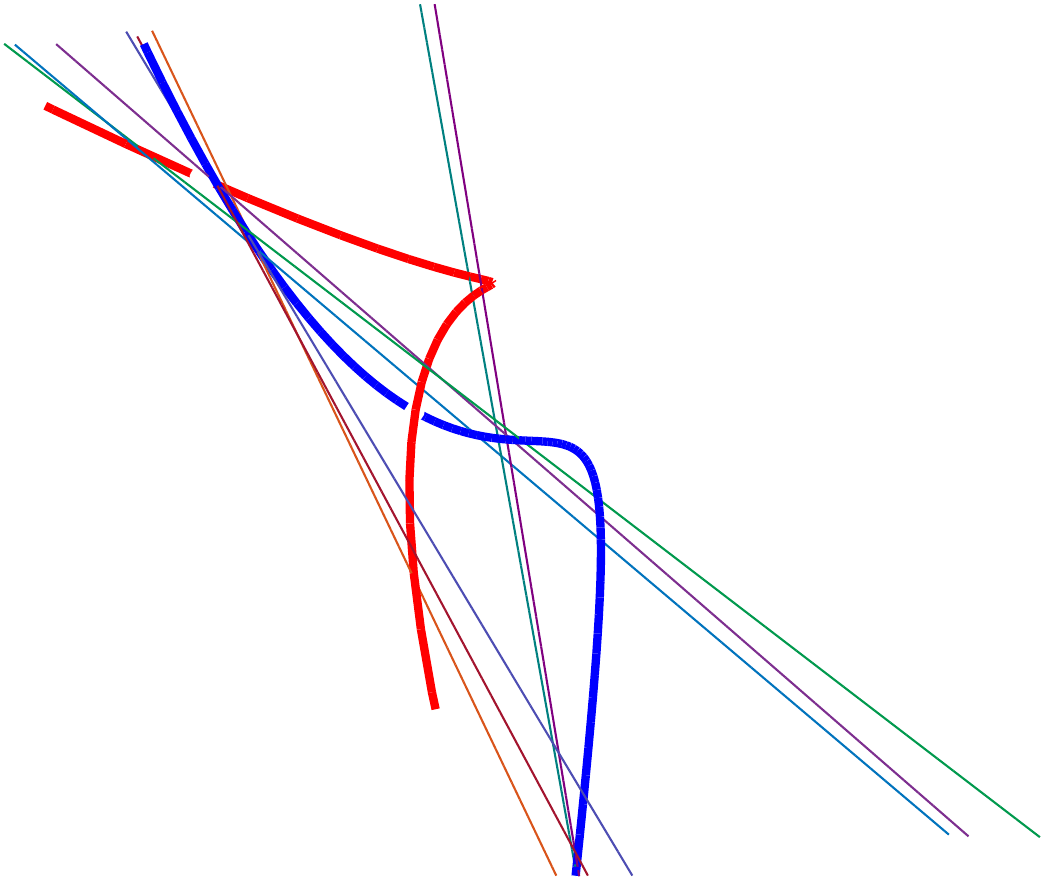}
\\[1ex]

$7$ & $(7,3,0)$ &
$\Msub
 { 1.5875 & 1.0292 &  2.9179 &  8.5023}
 {-2.2410 & 3.8004 & 16.0167 &  6.6939}
 { 1.7921 & 0.5681 &  3.0176 & 10.5398}$
&
\SecantImage{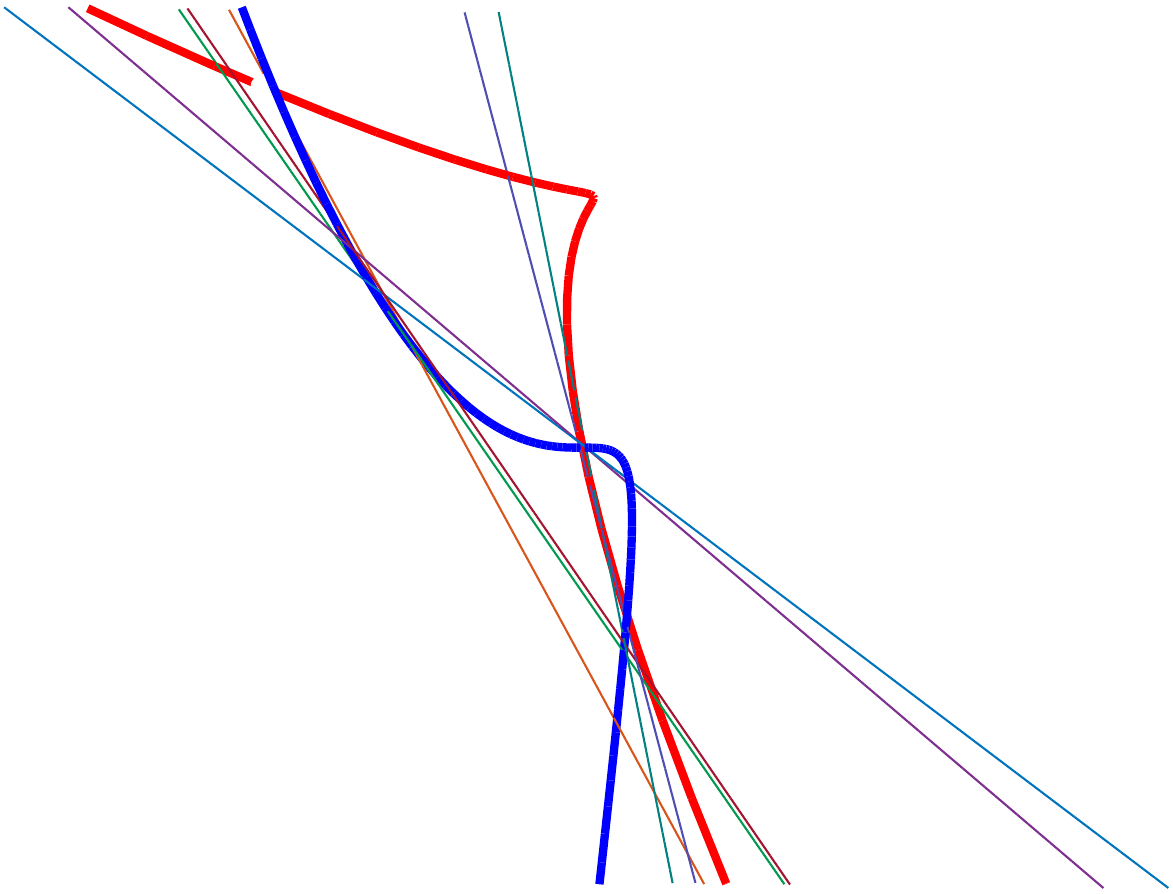}
\\[1ex]

$6$ & $(6,0,0)$ &
$\Msub
 {1.0321 & 1.8134 &  4.5132 &  7.8117}
 {0.3835 & 2.8268 & 14.8081 &  5.2579}
 {1.4025 & 2.0517 &  6.1857 & 10.9519}$
&
\SecantImage{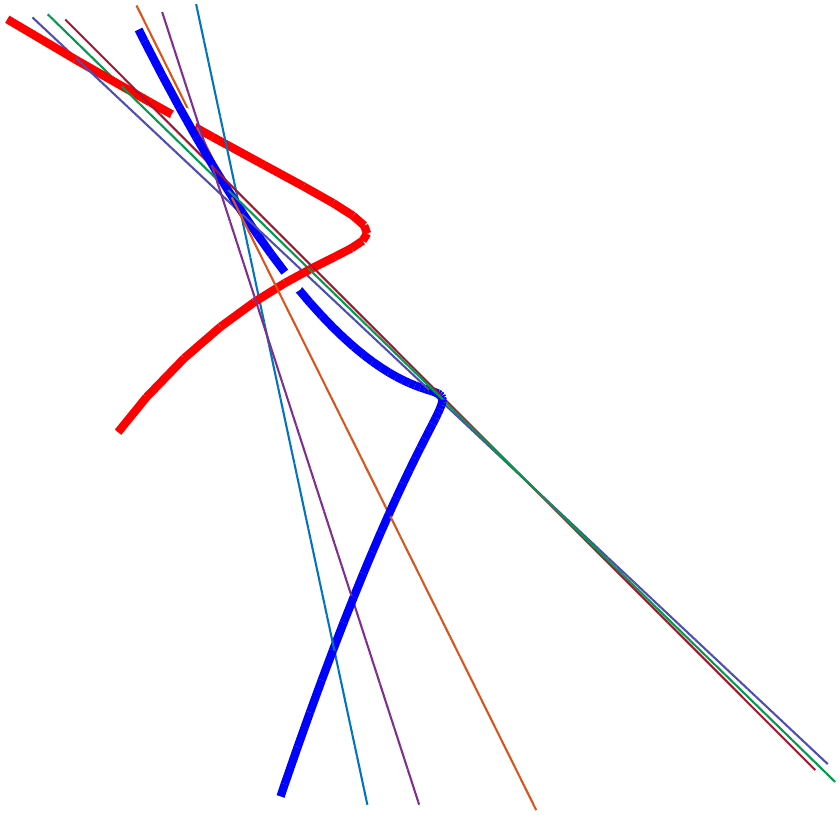}
\\[1ex]

$5$ & $(5,1,0)$ &
$\Msub
 { 2.0150 & 1.9005 &  2.2737 &  9.2034}
 {-0.0362 & 2.0729 & 14.6727 &  7.2270}
 { 1.8364 & 1.7414 &  4.0076 & 10.9568}$
&
\SecantImage{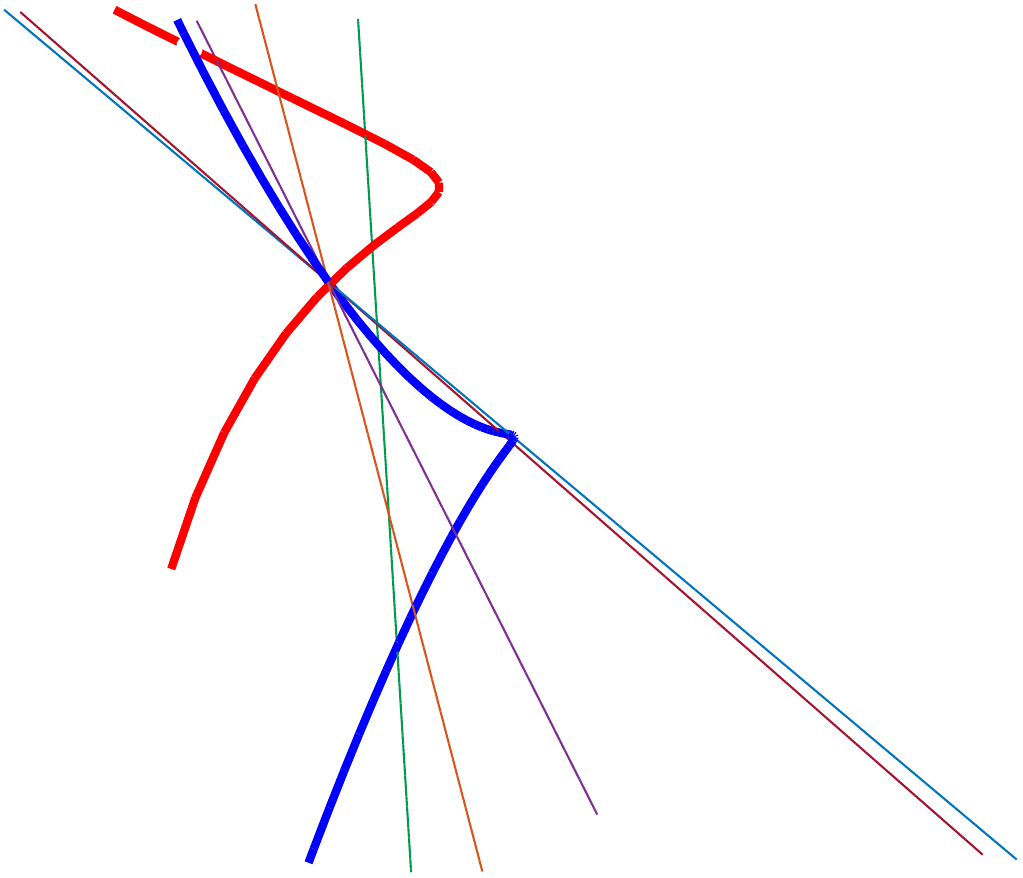}
\\[1ex]

$4$ & $(4,1,1)$ &
$\Msub
 { 0.2256 & 1.2068 &  5.2039 & 7.9386}
 {-1.1289 & 3.3178 & 14.3586 & 5.7655}
 { 1.1670 & 3.1542 &  2.6854 & 9.8156}$
&
\SecantImage{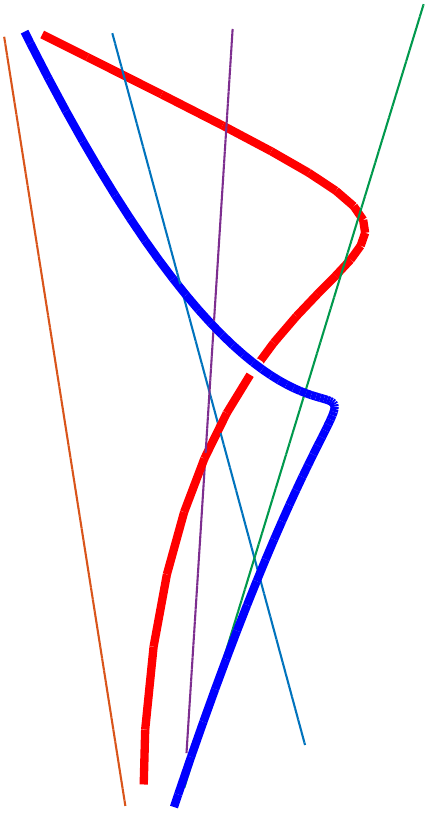}
\\[1ex]

$3$ & $(3,0,1)$ &
$\Msub
 {-0.7792 & 0.4435 &  5.0652 & 9.1369}
 { 1.2123 & 3.6701 & 12.5914 & 7.0705}
 {-0.2896 & 2.7162 &  3.8465 & 9.6116}$
&
\SecantImage{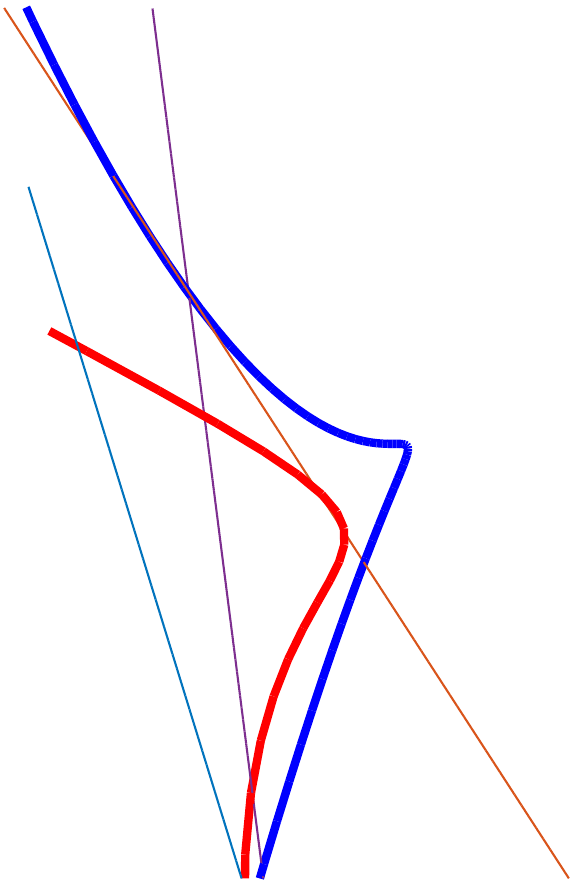}
\\[1ex]

$2$ & $(2,0,0)$ &
$\Msub
 { 1.7547 & 3.4027 &  1.7429 & 7.6041}
 {-1.0823 & 4.1538 & 12.0008 & 6.6773}
 { 2.6459 & 2.3080 &  3.5738 & 8.8431}$
&
\SecantImage{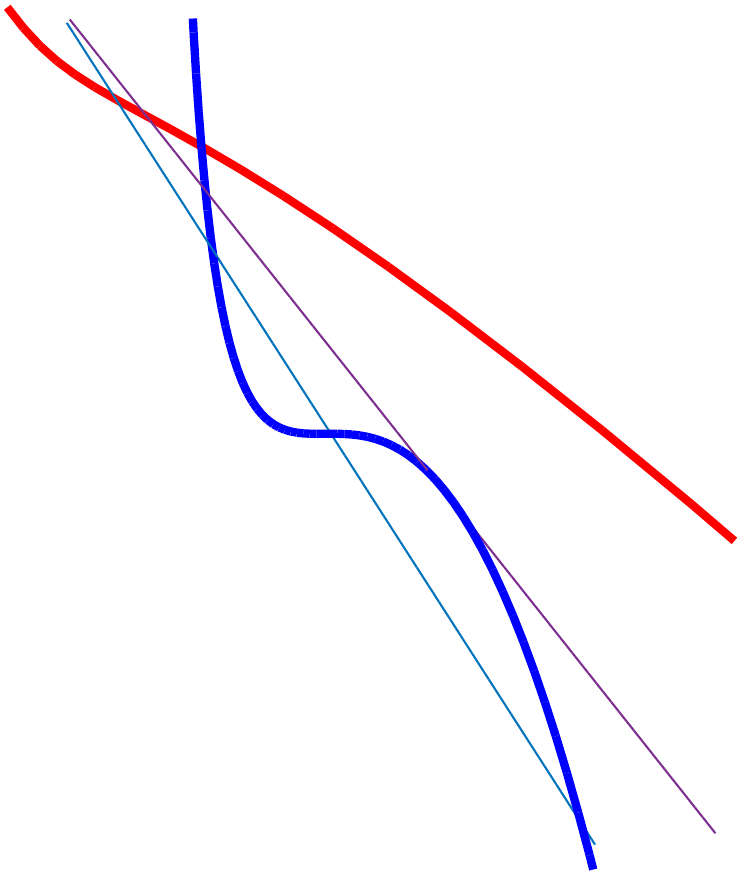}
\\[1ex]

$1$ & $(1,0,3)$ &
$\Msub
 { 0.0983 & 4.0479 &  2.2629 &  6.9867}
 {-0.0267 & 4.2539 & 12.6067 &  5.2851}
 {-1.3234 & 2.3696 &  3.5350 & 10.0737}$
&
\SecantImage{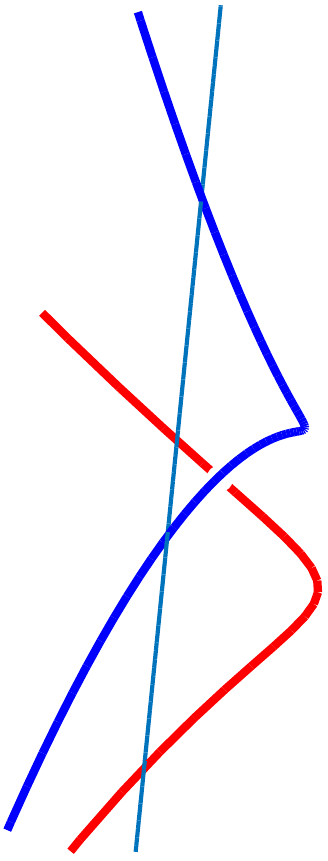}
\\[1ex]

$0$ & $(0,0,0)$ &
$\Msub
 { 1.6872 & 1.5909 &  2.7170 & 7.8740}
 {-1.0519 & 3.4185 & 14.1275 & 6.8633}
 { 2.7208 & 1.3832 &  4.2118 & 6.8602}$
&
\SecantImage{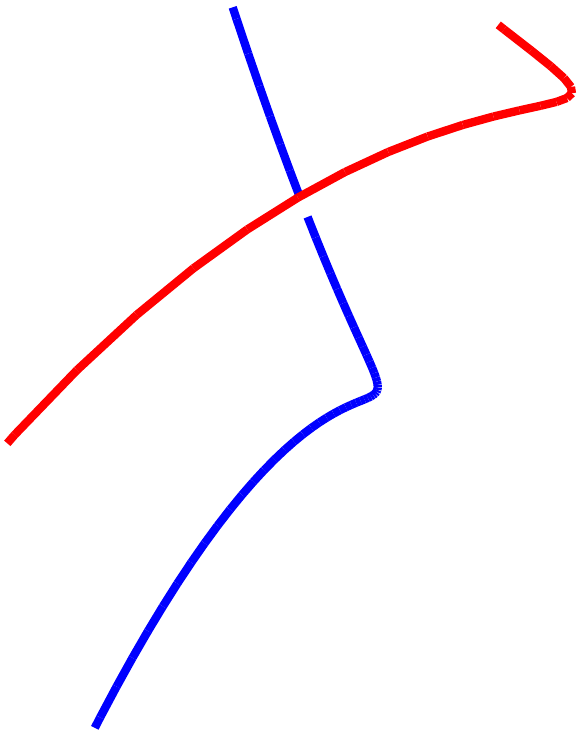}
\\[1ex]

\bottomrule
\end{longtable}

\clearpage

\section{Tables of 3-tuples}
\label{sec:Table}

\newcounter{tupleindex}
\newcommand{\numbertuple}[1]{%
  \stepcounter{tupleindex}%
  \arabic{tupleindex}.\,\ensuremath{#1}%
}

\begin{center}
\captionsetup{hypcap=false}
\captionof{table}{Realized 3-tuples, and admissible
but not realized 3-tuples.}
\label{tab:3-tuples}

\setlength{\tabcolsep}{2pt}
\renewcommand{\arraystretch}{1.2}

\setcounter{tupleindex}{0}

\begin{tabular}{@{}c@{\hspace{1.5em}}c@{}}

\begin{tabular}[t]{@{}llll@{}}
\toprule
\multicolumn{4}{c}{Realized} \\
\midrule
\numbertuple{$(0,0,0)$} & \numbertuple{$(0,0,2)$} & \numbertuple{$(0,0,4)$} & \numbertuple{$(0,0,6)$} \\
\numbertuple{$(0,0,8)$} & \numbertuple{$(0,1,1)$} & \numbertuple{$(0,1,3)$} & \numbertuple{$(0,1,5)$} \\
\numbertuple{$(0,1,7)$} & \numbertuple{$(0,2,0)$} & \numbertuple{$(0,2,2)$} & \numbertuple{$(0,2,4)$} \\
\numbertuple{$(0,2,6)$} & \numbertuple{$(0,3,1)$} & \numbertuple{$(0,3,3)$} & \numbertuple{$(0,3,5)$} \\
\numbertuple{$(0,4,0)$} & \numbertuple{$(0,4,2)$} & \numbertuple{$(0,4,4)$} & \numbertuple{$(0,5,1)$} \\
\numbertuple{$(0,5,3)$} & \numbertuple{$(0,6,0)$} & \numbertuple{$(0,6,2)$} & \numbertuple{$(0,7,1)$} \\
\numbertuple{$(0,8,0)$} & \numbertuple{$(1,0,1)$} & \numbertuple{$(1,0,3)$} & \numbertuple{$(1,0,5)$} \\
\numbertuple{$(1,1,0)$} & \numbertuple{$(1,1,2)$} & \numbertuple{$(1,1,4)$} & \numbertuple{$(1,2,1)$} \\
\numbertuple{$(1,2,3)$} & \numbertuple{$(1,2,5)$} & \numbertuple{$(1,3,0)$} & \numbertuple{$(1,3,2)$} \\
\numbertuple{$(1,3,4)$} & \numbertuple{$(1,4,1)$} & \numbertuple{$(1,4,3)$} & \numbertuple{$(1,5,0)$} \\
\numbertuple{$(1,5,2)$} & \numbertuple{$(1,6,1)$} & \numbertuple{$(1,7,0)$} & \numbertuple{$(2,0,0)$} \\
\numbertuple{$(2,0,2)$} & \numbertuple{$(2,0,4)$} & \numbertuple{$(2,0,6)$} & \numbertuple{$(2,0,8)$} \\
\numbertuple{$(2,1,1)$} & \numbertuple{$(2,1,3)$} & \numbertuple{$(2,1,5)$} & \numbertuple{$(2,1,7)$} \\
\numbertuple{$(2,2,0)$} & \numbertuple{$(2,2,2)$} & \numbertuple{$(2,2,4)$} & \numbertuple{$(2,2,6)$} \\
\numbertuple{$(2,3,1)$} & \numbertuple{$(2,3,3)$} & \numbertuple{$(2,3,5)$} & \numbertuple{$(2,4,0)$} \\
\numbertuple{$(2,4,2)$} & \numbertuple{$(2,4,4)$} & \numbertuple{$(2,5,1)$} & \numbertuple{$(2,5,3)$} \\
\numbertuple{$(2,6,0)$} & \numbertuple{$(3,0,1)$} & \numbertuple{$(3,0,3)$} & \numbertuple{$(3,0,5)$} \\
\numbertuple{$(3,0,7)$} & \numbertuple{$(3,1,0)$} & \numbertuple{$(3,1,2)$} & \numbertuple{$(3,1,4)$} \\
\numbertuple{$(3,1,6)$} & \numbertuple{$(3,2,1)$} & \numbertuple{$(3,2,3)$} & \numbertuple{$(3,2,5)$} \\
\numbertuple{$(3,3,0)$} & \numbertuple{$(3,3,2)$} & \numbertuple{$(3,3,4)$} & \numbertuple{$(3,4,1)$} \\
\numbertuple{$(3,4,3)$} & \numbertuple{$(3,5,0)$} & \numbertuple{$(4,0,0)$} & \numbertuple{$(4,0,2)$} \\
\numbertuple{$(4,0,4)$} & \numbertuple{$(4,1,1)$} & \numbertuple{$(4,1,3)$} & \numbertuple{$(4,1,5)$} \\
\numbertuple{$(4,2,0)$} & \numbertuple{$(4,2,2)$} & \numbertuple{$(4,2,4)$} & \numbertuple{$(4,3,1)$} \\
\numbertuple{$(4,3,3)$} & \numbertuple{$(4,4,0)$} & \numbertuple{$(5,0,1)$} & \numbertuple{$(5,0,3)$} \\
\numbertuple{$(5,0,5)$} & \numbertuple{$(5,1,0)$} & \numbertuple{$(5,1,2)$} & \numbertuple{$(5,1,4)$} \\
\numbertuple{$(5,2,1)$} & \numbertuple{$(5,2,3)$} & \numbertuple{$(5,3,0)$} & \numbertuple{$(5,3,2)$} \\
\numbertuple{$(5,4,1)$} & \numbertuple{$(5,5,0)$} & \numbertuple{$(6,0,0)$} & \numbertuple{$(6,0,2)$} \\
\numbertuple{$(6,0,4)$} & \numbertuple{$(6,1,1)$} & \numbertuple{$(6,1,3)$} & \numbertuple{$(6,2,0)$} \\
\numbertuple{$(6,2,2)$} & \numbertuple{$(6,3,1)$} & \numbertuple{$(6,4,0)$} & \numbertuple{$(7,0,1)$} \\
\numbertuple{$(7,0,3)$} & \numbertuple{$(7,1,0)$} & \numbertuple{$(7,1,2)$} & \numbertuple{$(7,2,1)$} \\
\numbertuple{$(7,3,0)$} & \numbertuple{$(8,0,0)$} & \numbertuple{$(8,0,2)$} & \numbertuple{$(8,1,1)$} \\
\numbertuple{$(8,2,0)$} & \numbertuple{$(9,0,1)$} & \numbertuple{$(9,1,0)$} & \numbertuple{$(10,0,0)$} \\
\bottomrule
\end{tabular}

&

\begin{tabular}[t]{@{}lll@{}}
\toprule
\multicolumn{3}{c}{Admissible but not realized} \\
\midrule
\numbertuple{$(0,0,10)$} & \numbertuple{$(0,1,9)$} & \numbertuple{$(0,2,8)$} \\
\numbertuple{$(0,3,7)$} & \numbertuple{$(0,4,6)$} & \numbertuple{$(0,5,5)$} \\
\numbertuple{$(0,6,4)$} & \numbertuple{$(0,7,3)$} & \numbertuple{$(0,8,2)$} \\
\numbertuple{$(0,9,1)$} & \numbertuple{$(0,10,0)$} & \numbertuple{$(1,0,9)$} \\
\numbertuple{$(1,0,7)$} & \numbertuple{$(1,1,8)$} & \numbertuple{$(1,1,6)$} \\
\numbertuple{$(1,2,7)$} & \numbertuple{$(1,3,6)$} & \numbertuple{$(1,4,5)$} \\
\numbertuple{$(1,5,4)$} & \numbertuple{$(1,6,3)$} & \numbertuple{$(1,7,2)$} \\
\numbertuple{$(1,8,1)$} & \numbertuple{$(1,9,0)$} & \numbertuple{$(2,6,2)$} \\
\numbertuple{$(2,7,1)$} & \numbertuple{$(2,8,0)$} & \numbertuple{$(3,5,2)$} \\
\numbertuple{$(3,6,1)$} & \numbertuple{$(3,7,0)$} & \numbertuple{$(4,0,6)$} \\
\numbertuple{$(4,4,2)$} & \numbertuple{$(4,5,1)$} & \numbertuple{$(4,6,0)$} \\
\bottomrule
\end{tabular}
\\
\end{tabular}
\end{center}

\end{document}